\documentclass{amsart}

\usepackage{amsrefs}
\usepackage{amssymb}
\usepackage{amsaddr}
\usepackage{tikz}

\newtheorem{theorem}[equation]{Theorem}
\newtheorem{lemma}[equation]{Lemma}

\newtheorem{claim}[equation]{Claim}

\newtheorem{definition}[equation]{Definition}
\newtheorem{remark}[equation]{Remark}
\newtheorem{problem}[equation]{Problem}

\newcommand{\eref}[1]{(\ref{e.#1})}
\newcommand{\tref}[1]{Theorem \ref{t.#1}}
\newcommand{\lref}[1]{Lemma \ref{l.#1}}
\newcommand{\cref}[1]{Corollary \ref{c.#1}}
\newcommand{\clref}[1]{Claim \ref{cl.#1}}
\newcommand{\fref}[1]{Figure \ref{f.#1}}
\newcommand{\sref}[1]{Section \ref{s.#1}}

\numberwithin{equation}{section}
\numberwithin{figure}{section}

\newcommand{\Z}{\mathbb{Z}}
\newcommand{\R}{\mathbb{R}}

\renewcommand{\P}{\mathbb{P}}

\newcommand{\ep}{\varepsilon}
\newcommand{\id}{1}
\newcommand{\trace}{\operatorname{trace}}
\newcommand{\dist}{\operatorname{dist}}

\begin{document}

\title[Anderson localization]{Localization near the edge for the Anderson Bernoulli model on the two dimensional lattice}

\author{Jian Ding}
\address{University of Pennsylvania}
\email{dingjian@wharton.upenn.edu}

\author{Charles K Smart}
\address{University of Chicago}
\email{smart@math.uchicago.edu}

\begin{abstract}
We consider a Hamiltonian given by the Laplacian plus a Bernoulli potential on the two dimensional lattice.  We prove that, for energies sufficiently close to the edge of the spectrum, the resolvent on a large square is likely to decay exponentially.  This implies almost sure Anderson localization for energies sufficiently close to the edge of the spectrum.  Our proof follows the program of Bourgain--Kenig, using a new unique continuation result inspired by a Liouville theorem of Buhovsky--Logunov--Malinnikova--Sodin.
\end{abstract}

\maketitle

\section{Introduction}

\subsection{Anderson localization}

We consider the Anderson--Bernoulli model on the lattice, which is the random Schr\"odinger operator on $\ell^2(\Z^d)$ given by
\begin{equation*}
H = - \Delta + \delta V,
\end{equation*}
where $(\Delta u)(x) = \sum_{|y-x| = 1} (u(y) - u(x))$ is the discrete Laplacian, $(V u)(x) = V_x u(x)$ is a random potential whose values $V_x \in \{ 0, 1 \}$ for $x \in \Z^d$ are independent and satisfy $\P[V_x = 0] = \P[V_x = 1] = 1/2$, and $\delta > 0$ is the strength of the noise.

We are interested in the effect of the perturbation $\delta V$ on the spectral theory of the discrete Laplacian.  We recall that the spectrum of the discrete Laplacian is the closed interval $\sigma(-\Delta) = [0,4d]$, as can be seen by taking the Fourier transform, $\widehat{-\Delta}(\xi) = 2 \sum_{1 \leq k \leq d} (1-\cos(\xi_k)).$  We recall that the spectrum of the random Hamiltonian is almost surely the closed interval $\sigma(H) = [0,4d+\delta]$, as can be seen by observing that, almost surely, every finite configuration appears in the random Bernoulli potential.  While the spectrum of the discrete Laplacian is absolutely continuous, the random Hamiltonian may have eigenvalues.  The perturbation $\delta V$ can create ``traps'' on which eigenfunctions are exponentially localized.  This phenomenon is called Anderson localization.

To be more precise, we say that $H$ has ``Anderson localization'' in the spectral interval $I \subseteq \sigma(H)$ if
\begin{equation*}
\psi : \Z^d \to \R, \quad
\lambda \in I, \quad
H \psi = \lambda \psi, \quad \mbox{and} \quad
\inf_{n > 0} \sup_{x \in \Z^d} (1 + |x|)^{-n} |\psi(x)| < \infty
\end{equation*} 
implies
\begin{equation*}
\inf_{t > 0} \sup_{x \in \Z^d} e^{t |x|} |\psi(x)| < \infty.
\end{equation*}
That is, $H$ has Anderson localization if every polynomially bounded solution of the eigenfunction equation is in fact an exponentially decaying eigenfunction.  Recall (see for example Kirsch \cite{Kirsch}*{Section 7}) that Anderson localization in $I$ implies that the spectrum of $H$ in $I$ is pure point.

We prove the following result.

\begin{theorem}
\label{t.anderson2d}
In dimension $d = 2$ there is an $\ep > 0$, depending on $\delta > 0$, such that, almost surely, $H$ has Anderson localization in $[0,\ep]$.
\end{theorem}

\subsection{Background}

To put \tref{anderson2d} in context, let us very briefly discuss some of the known results and open problems for the Anderson--Bernoulli model.  We refer the reader who is interested in more background material to Aizenman--Warzel \cite{Aizenman-Warzel}, Hundertmark \cite{Hundertmark}, Jitomirskaya \cite{Jitomirskaya}, Kirsch \cite{Kirsch}, and Stolz \cite{Stolz}.

In the case the noise is continuous, that is, when we replace the Bernoulli random variables $V_x$ with some other random variables $V_x$ which are independent and share a compactly supported and bounded density, the following are true:

\begin{itemize}
\item If $d = 1$, then $H$ almost surely has Anderson localization in all of $\sigma(H)$.
\item If $d \geq 2$, then $H$ almost surely has Anderson localization in $[0,\ep]$.
\item If $d \geq 2$ and $\delta \geq C$ is large, then $H$ almost surely has Anderson localization in all of $\sigma(H)$.
\end{itemize}
The above results can be proved using the spectral averaging method of Kunz--Souillard \cite{Kunz-Souillard}, the multiscale method of Fr\"ohlich--Spencer \cite{Frohlich-Spencer}, or the fractional moment method of Aizenman--Molchanov \cite{Aizenman-Molchanov}.  The historical development is fairly complicated, and we refer the interested reader to Kirsch \cite{Kirsch}*{Section 8 Notes} for an account and references.

The above three results do not pin down what happens when $d \geq 2$ and $\delta > 0$ is small.  This unknown regime is the subject of the two most important open problems in the area; see Simon \cite{Simon}.

\begin{problem}
In dimension $d = 2$, prove that $H$ almost surely has Anderson localization in all of $\sigma(H)$. 
\end{problem}

\begin{problem}
In dimensions $d \geq 3$, prove that, for every small $\ep > 0$, there is a $\delta > 0$ such that $H$ almost surely has no eigenvalues in $[\ep,4d-\ep]$.
\end{problem}

In addition to the above two conjectures, it is expected that localization effects are quite robust to changes in the structure of the noise.  That is, Anderson localization is expected to be a universal phenomenon.  However, when the law of $V_x$ does not have a bounded density, such as in the Bernoulli case, much less is known.  In dimension $d = 1$, Carmona--Klein--Martinelli \cite{Carmona-Klein-Martinelli} proved almost sure Anderson localization in all of $\sigma(H)$ for any non-trivial i.i.d potential.  In dimension $d \geq 2$, when the law of $V_x$ is a ``sufficiently nice'' approximation of a uniform random variable on $[0,1]$, Imbrie \cite{Imbrie} proved almost sure Anderson localization in $[0,\ep]$.  If one works on the continuum $\R^d$ instead of the lattice $\Z^d$, then Bourgain--Kenig \cite{Bourgain-Kenig} prove almost sure Anderson localization in $[0,\ep]$.  Our paper follows these previous works, giving more rigorous evidence of the universality of Anderson localization.

\subsection{Resolvent estimate}

We do not prove \tref{anderson2d} directly.  Instead, we rely on previous work to reduce the problem to proving bounds on the exponential decay of the resolvent.  Our main theorem is the following.

\begin{theorem}
\label{t.main}
Suppose $d = 2$ and $\delta = 1$.  For any $1/2 > \gamma > 0$, there are $\alpha > 1 > \ep > 0$ such that, for every energy $\bar \lambda \in [0,\ep]$ and square $Q \subseteq \Z^2$ of side length $L \geq \alpha$,
\begin{equation*}
\P[|(H_Q - \bar \lambda)^{-1}(x,y)| \leq e^{L^{1-\ep} - \ep |x - y|} \mbox{ for } x, y \in Q] \geq 1 - L^{-\gamma}.
\end{equation*}
\end{theorem}

Here $H_Q : \ell^2(Q) \to \ell^2(Q)$ denotes the restriction of the Hamiltonian $H$ to the square $Q$ with zero boundary conditions.  

Our proof works, essentially verbatim, for any random potential $V : \Z^2 \to \R$ whose values $V(x)$ are i.i.d., bounded, and non-trivial.  However, for simplicity, we argue only in the case of strength $\delta = 1$ and $\frac12$-Bernoulli noise.

\begin{proof}[Proof of \tref{anderson2d}]
Almost sure Anderson localization for $H$ in the interval $[0,\ep]$ follows from \tref{main} using the Peierls argument of Bourgain--Kenig \cite{Bourgain-Kenig}*{Section 7}.  See Germinet--Klein \cite{Germinet-Klein}*{Sections 6 and 7} for an axiomatic version of this.
\end{proof}

\subsection{Outline}

At a high level, our proof follows Bourgain--Kenig \cite{Bourgain-Kenig}.  We perform a multiscale analysis, keeping track of a list of ``frozen'' sites $F \subseteq \Z^2$ where the potential has already been sampled.  The complementary ``free'' sites $\Z^2 \setminus F$ are sampled only to perform an eigenfunction variation on rare, ``bad'' squares.  This strategy of frozen and free sites is used to obtain a version of the Wegner \cite{Wegner} estimate that is otherwise unavailable in the Bernoulli setting.

The eigenvalue variation of Bourgain--Kenig \cite{Bourgain-Kenig} relies crucially on an a priori quantitative unique continuation result.  Namely, for every $\alpha > 1$ there is a $\beta > 1$ such that, if $u \in C^2(B_R)$ satisfies $$|\Delta u| \leq \alpha |u| \leq \alpha^2 |u(0)| \quad \mbox{in } B_R,$$ then $$\max_{y \in B_1(x)} |u(y)| \geq e^{-\beta R^{4/3} \log R} |u(0)| \quad \mbox{for } B_1(x) \subseteq B_{R/2}.$$ The corresponding fact (even in qualitative form) is false on the lattice $\Z^d$; see Jitomirskaya \cite{Jitomirskaya}*{Theorem 2}.

To carry out the program on the lattice, we need a substitute for the above quantitive unique continuation result.  For the two-dimensional lattice $\Z^2$, a hint of the missing ingredient appears in the paper of Buhovsky--Logunov--Malinnikova--Sodin \cite{Buhovsky-Logunov-Malinnikova-Sodin}.  In this paper, it is proved that any function $u : \Z^2 \to \R$ that is harmonic and bounded on a $1-\ep$ fraction of sites must be constant.  One of the key components of this Liouville theorem is the following quantitative unique continuation result for harmonic functions on the two dimensional lattice.

\begin{theorem}[\cite{Buhovsky-Logunov-Malinnikova-Sodin}]
There are constants $\alpha > 1 > \ep > 0$ such that, if $u : \Z^2 \to \R$ is lattice harmonic in a square $Q \subseteq \Z^2$ of side length $L \geq \alpha$, then
\begin{equation*}
| \{ x \in Q : |u(x)| \geq e^{-\alpha L} \| u \|_{\ell^\infty(\frac12 Q)} \} | \geq \ep L^2.
\end{equation*}
\end{theorem}

This implies that any two lattice harmonic functions that agree on a $1-\ep$ fraction of sites in a large square must be equal in the concentric half square.  Note that this result is false in dimensions three and higher.

Inspired by this theorem and its proof, we prove the following random quantitative unique continuation result for eigenvalues of the Hamiltonian $H$.

\begin{theorem}
\label{t.introcontinuation}
There are constants $\alpha > 1 > \ep > 0$ such that, if $\bar \lambda \in [0,9]$ is an energy and $Q \subseteq \Z^2$ is a square of side length $L \geq \alpha$, then $\P[\mathcal E] \geq 1 - e^{-\ep L^{1/4}}$, where $\mathcal E$ denotes the event that
\begin{equation*}
| \{ x \in Q : |\psi(x)| \geq e^{-\alpha L \log L} \| \psi \|_{\ell^\infty(\frac12 Q)} \} | \geq \ep L^{3/2}(\log L)^{-1/2}
\end{equation*}
holds whenever $\lambda \in \R$, $\psi : \Z^2 \to \R$, $|\lambda - \bar \lambda| \leq e^{-\alpha (L \log L)^{1/2}}$, and $H \psi = \lambda \psi$ in $Q$.
\end{theorem}

This is one of three main contributions of our work.  Roughly speaking, this result says that, with high probability, every eigenfunction on a  square $Q$ with side length $L$ is supported on at least $L^{3/2-\ep}$ many points in $Q$. We in fact prove something slightly stronger, as our unique continuation result needs to be adapted to the ``frozen'' and ``free'' sites formalism.  See \tref{continuation} below.

In analogy with the Wegner estimate for continuous noise, we expect that, with probability $1 - e^{-L^{1-\ep}}$, there are no $\psi : \Z^d \to \R$ satisfying $H \psi = \lambda \psi$ in $Q$, $|\lambda - \bar \lambda| < e^{-L^{1-\ep}}$, and $\psi = 0$ on $\Z^d \setminus Q$.  That is, we expect the above unique continuation theorem to be vacuous in the case of Dirichlet data.  Of course, there is (as of our writing) no such Wegner estimate available in the Bernoulli case.  Moreover, we apply this result below for $\psi$ with non-zero boundary data.  Still, it is worth keeping in mind that our unique continuation theorem is quite weak, and barely suffices for our application to Anderson localization.

The sparsity of support in our unique continuation theorem forces us to make two significant modifications to Bourgain--Kenig \cite{Bourgain-Kenig} program.  First, since the classical Sperner's lemma no longer suffices for eigenvalue variation, we prove a generalization of Sperner's theorem; see \tref{sperner} below.  Second, since we have to deal with eigenvalue interlacements during the eigenvalue variation, we are forced to identify and exploit gaps in the spectrum; see \lref{minmax} and \lref{wegner} below.  These two modifications are the other two main contributions of our work.

All of the essentially new ideas in this article are presented in the third, fourth, and fifth sections.  The remaining three sections consist of relatively straightforward modifications of the ideas in Bourgain--Kenig \cite{Bourgain-Kenig}.  However, we did invest some effort in reorganizing these latter arguments.  Our proofs of the multiscale resolvent estimate \lref{multiscale} and the base case \lref{principal} appear to be new.

Recently, our modifications to the multiscale analysis were used to prove the analogue of \tref{main} in dimension $d = 3$.  Li--Zhang \cite{Li-Zhang} proved a deterministic unique continuation result that is a sufficient substitute for \tref{introcontinuation}.  We find it quite interesting that deterministic unique continuation suffices in dimension $d = 3$ while dimension $d = 2$ appears to require the use of the randomness.

\subsection*{Acknowledgments}

We thank Carlos Kenig for making us aware of the paper of Buhovsky--Logunov--Malinnikova--Sodin.  We thank Adrian Dietlein, Eugenia Malinnikova, and the anonymous referees for finding mistakes in earlier versions of the manuscript.  The first author was partially supported by the NSF award DMS-1757479 and a Sloan Foundation fellowship.  The second author was partially supported by the NSF award DMS-1712841.

\section{Preliminaries}

\subsection{Spectrum}

As described above, it is a standard fact that the spectrum of $H$ is almost surely the interval $[0,9]$.  Henceforth, we only concern ourselves with energies in this interval.  In particular, $\lambda$ always denotes a real number in the interval $[0,9]$.  Moreover, we fix a target energy
\begin{equation*}
\bar \lambda \in [0,9]
\end{equation*}
throughout the article.

\subsection{Continuous variables}

While we are proving a theorem about Bernoulli potentials $V : \Z^2 \to \{ 0, 1 \}$, it is useful to allow the potential to take values in the real interval $[0,1]$.  In particular, we take our probability space to be the set $V : \Z^2 \to [0,1]$ equipped with the usual Borel sets and a probability measure given by the product of the $\frac12$-Bernoulli measure.  This formalism is borrowed from Bourgain--Kenig \cite{Bourgain-Kenig} and is used in exactly one place: to control the number of eigenvalues during the proof of the Wegner estimate in \lref{wegner}.  One way to think about this is to work instead with a non-trivial convex combination of the Bernoulli and uniform potential and prove estimates that are independent of the choice of weights.

\subsection{Squares}

Unless otherwise specified, the letter $Q$ denotes a dyadic square in $\Z^2$.  That is, a set
\begin{equation*}
Q = x + [0,2^n)^2 \cap \Z^2 \quad \mbox{with } x \in \Z^2.
\end{equation*}
The side length and area of $Q$ are $\ell(Q) = 2^n$ and $|Q| = \ell(Q)^2 = 2^{2n}$.  The notations $\tfrac12 Q$ and $2 Q$ denote the concentric halving and doubling of $Q$, respectively.

\subsection{Restrictions to finite sets}

We frequently consider the restriction $H_Q = \id_Q H \id_Q$ of the Hamiltonian $H$ to squares $Q \subseteq \Z^2$.  We use the notation $R_Q = (H_Q - \bar \lambda)^{-1}$ to indicate (when it exists) the unique operator on $\ell(\Z^2)$ such that $R_Q = \id_Q R_Q \id_Q$ and $R_Q (H_Q - \bar \lambda) = (H_Q - \bar \lambda) R_Q = \id_Q$.  Abusing notation, we sometimes think of $H_Q$ and $R_Q$ as elements of the space $S^2(\R^Q)$ of symmetric bilinear forms on $\R^Q$.  Similarly, we sometimes think of the restriction $V_Q$ as an element of the vector space $\R^Q$.

\subsection{Notation}

We use Hardy notation for constants, letting $C > 1 > c > 0$ denote universal constants that may differ in each instance.  We use subscipts to denote additional dependencies, so that $C_\ep$ is allowed to depend on $\ep$.

We use $\| H_Q \|$ and $\| H_Q \|_2$ to denote the operator and Hilbert--Schmidt norms of $H_Q \in S^2(\R^Q)$.  For functions $\psi \in \R^Q$, we make frequent use of the bounds $\| \psi \|_{\ell^\infty(Q)} \leq \| \psi \|_{\ell^2(Q)} \leq |Q|^{1/2} \| \psi \|_{\ell^\infty(Q)}$ to absorb differences of norms into exponential prefactors.

When $\psi : Q \to \R$ and $t \in \R$, we use $\{ |\psi| \geq t \}$ as shorthand for the set $\{ x \in Q : |\psi(x)| \geq t \}$.

We occasionally use notation from functional calculus.  In particular, if $A \in S^2(\R^n)$ and $I \subseteq \R$, then $\trace \id_I(A)$ is the number of eigenvalues of $A$ in the interval $I$.

\section{Unique continuation with a random potential}

\subsection{Statement}

We prove a quantitative unique continuation result for eigenfunctions of $H$.  Our argument generalizes the unique continuation result of  \cite{Buhovsky-Logunov-Malinnikova-Sodin} for harmonic functions on $\Z^2$.  The basic idea is that, with high probability, every eigenfunction in the square $Q$ is supported on $\ell(Q)^{3/2-\ep}$ many sites.  The precise statement is made more complicated by the presence of frozen sites, which we need in our application to Anderson localization.

In order to state our result precisely, we need to define the $45^\circ$ rotations of rectangles and lines.

\begin{definition}
A tilted rectangle is a set $$R_{I,J} = \{ (x,y) \in \Z^2 : x+y \in I \mbox{ and } x-y \in J \},$$ where $I, J \subseteq \Z$ are intervals.  A tilted square $Q$ is a tilted rectangle $R_{I,J}$ with side length $\ell(Q) = |I| = |J|$.
\end{definition}

\begin{definition}
Given $k \in \Z$, define the diagonals $$D^\pm_k = \{ (x,y) \in \Z^2 : x \pm y = k \}.$$  
\end{definition}

We need a notion of sparsity along diagonals.

\begin{definition}
Suppose $F \subseteq \Z^2$ a set, $\delta > 0$ a density, and $R$ a tilted rectangle.  Say that $F$ is $(\delta,\pm)$-sparse in $R$ if $$|D_k^\pm \cap F \cap R| \leq \delta |D_k^\pm \cap R| \quad \mbox{for all diagonals } D_k^\pm.$$  We say that $F$ is $\delta$-sparse in $R$ if it is both $(\delta,+)$-sparse and $(\delta,-)$-sparse in $R$.
\end{definition}

We need a notion of sparsity at all scales.

\begin{definition}
Say that $F$ is $\delta$-regular in the set $E \subseteq \Z^2$ if $\sum_k |Q_k| \leq \delta |E|$ holds whenever $F$ is not $\delta$-sparse in each of the disjoint tilted squares $Q_1, ..., Q_n \subseteq E$.
\end{definition}

We now state our unique continuation theorem.  This is the same as \tref{introcontinuation}, except that it has been adapted to allow for a regular set of ``frozen'' sites.

\begin{theorem}
\label{t.continuation}
For every small $\ep > 0$, there is a large $\alpha > 1$ such that, if
\begin{enumerate}
\item $Q \subseteq \Z^2$ a square with $\ell(Q) \geq \alpha$
\item $F \subseteq Q$ is $\ep$-regular in $Q$
\item $v : F \to \{ 0, 1 \}$
\item $\mathcal E_{\mathrm{uc}}(Q,F)$ denotes the event that
\begin{equation*}
\begin{cases}
|\lambda - \bar \lambda| \leq e^{-\alpha (\ell(Q) \log \ell(Q))^{1/2}} \\
H \psi = \lambda \psi \mbox{ in } Q \\
|\psi| \leq 1 \mbox{ in a $1- \ep (\ell(Q) \log(\ell(Q)))^{-1/2}$ fraction of $Q \setminus F$},
\end{cases}
\end{equation*}
implies $|\psi| \leq e^{\alpha \ell(Q) \log \ell(Q)}$ in $\tfrac12 Q$,
\end{enumerate}
then $\P[\mathcal E_{\mathrm{uc}}(Q,F) | V_F = v] \geq 1 - e^{-\ep \ell(Q)^{1/4}}$.
\end{theorem}

The rest of this section is devoted to the proof of \tref{continuation}.

\subsection{Tilted coordinates}

We work in the tilted coordinates
\begin{equation*}
(s,t) = (x+y,x-y).
\end{equation*}
The lattice is $\{ (s,t) \in \Z^2 : s-t \mbox{ even} \}$.  The tilted rectangles are $R_{I,J} = \{ (s,t) \in I \times J : s - t \mbox{ even} \}.$  The equation $H \psi = \lambda \psi$ at the point $(s,t)$ is
\begin{equation}
\label{e.tilted1}
(4 + V_{s,t} - \lambda) \psi_{s,t} - \psi_{s-1,t-1} - \psi_{s+1,t+1} - \psi_{s-1,t+1} - \psi_{s+1,t-1} = 0.
\end{equation}

\subsection{Basic lemmas}

We recall and modify some elementary results from \cite{Buhovsky-Logunov-Malinnikova-Sodin}.  These give a priori bounds on how information propagates from the boundary to the interior of a tilted rectangle.  

\begin{definition}
  The west boundary of a tilted rectangle is
  \begin{equation*}
  \partial^w R_{[a,b],[c,d]} = R_{[a,a+1],[c,d]} \cup R_{[a,b],[c,c+1]}.
  \end{equation*}
\end{definition}

The main idea is that, if the equation $H \psi = \lambda \psi$ holds in a tilted rectangle $R$, then the values of $\psi$ on $R$ are determined by the values of $\psi$ on the west boundary $\partial^w R$.  A qualitative version of this is the following.

\begin{lemma}
\label{l.extension}
Every function $\psi : \partial^w R_{[1,a],[1,b]} \to \R$ has a unique extension $\psi : R_{[1,a],[1,b]} \to \R$ that satisfies $H \psi = \lambda \psi$ in $R_{[2,a-1],[2,b-1]}$. \qed
\end{lemma}

\begin{proof}
First, observe that the equation $H \psi = \lambda \psi$ at $(s-1,t-1)$ rearranges to
\begin{equation}
\label{e.tilted2}
\psi_{s,t} = (4 + V_{s-1,t-1} - \lambda) \psi_{s-1,t-1} - \psi_{s-2,t} - \psi_{s,t-2} - \psi_{s-2,t-2}.
\end{equation}
Second, observe that, if $(s,t) \in R_{[1,a],[1,b]} \setminus \partial^w R_{[1,a],[1,b]}$, then $(s-1,t-1) \in R_{[2,a-1],[2,b-1]}$ and $(s-2,t), (s,t-2), (s-2,t-2) \in R_{[1,a],[1,b]}$.  In particular, we can recursively iterate \eref{tilted2} for $(s-1,t-1) \in R_{[2,a-1],[2,b-1]}$ to determine the values of $\psi$ on $R$ from the values of $\psi$ in $\partial^w R$.  
\end{proof}

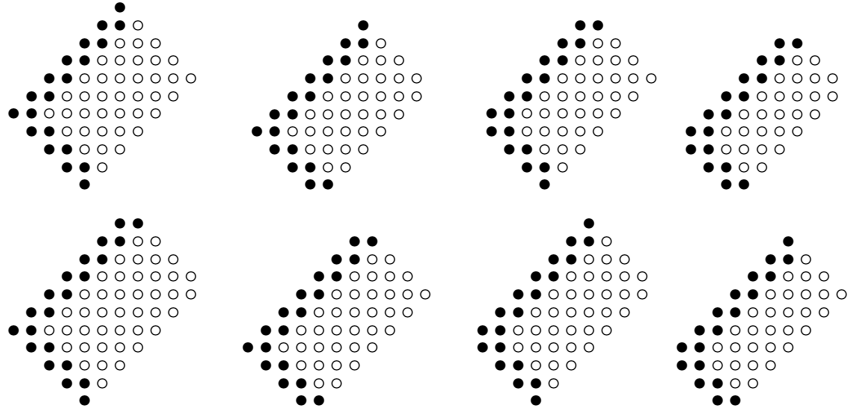
\begin{figure}
\begin{tabular}{cccc}
\begin{tikzpicture}[scale=1/6,rotate=45,yscale=-1]
\draw (1,1) node {$\bullet$};
\draw (1,3) node {$\bullet$};
\draw (1,5) node {$\bullet$};
\draw (1,7) node {$\bullet$};
\draw (1,9) node {$\bullet$};
\draw (2,2) node {$\bullet$};
\draw (2,4) node {$\bullet$};
\draw (2,6) node {$\bullet$};
\draw (2,8) node {$\bullet$};
\draw (3,1) node {$\bullet$};
\draw (3,3) node {$\circ$};
\draw (3,5) node {$\circ$};
\draw (3,7) node {$\circ$};
\draw (3,9) node {$\circ$};
\draw (4,2) node {$\bullet$};
\draw (4,4) node {$\circ$};
\draw (4,6) node {$\circ$};
\draw (4,8) node {$\circ$};
\draw (5,1) node {$\bullet$};
\draw (5,3) node {$\circ$};
\draw (5,5) node {$\circ$};
\draw (5,7) node {$\circ$};
\draw (5,9) node {$\circ$};
\draw (6,2) node {$\bullet$};
\draw (6,4) node {$\circ$};
\draw (6,6) node {$\circ$};
\draw (6,8) node {$\circ$};
\draw (7,1) node {$\bullet$};
\draw (7,3) node {$\circ$};
\draw (7,5) node {$\circ$};
\draw (7,7) node {$\circ$};
\draw (7,9) node {$\circ$};
\draw (8,2) node {$\bullet$};
\draw (8,4) node {$\circ$};
\draw (8,6) node {$\circ$};
\draw (8,8) node {$\circ$};
\draw (9,1) node {$\bullet$};
\draw (9,3) node {$\circ$};
\draw (9,5) node {$\circ$};
\draw (9,7) node {$\circ$};
\draw (9,9) node {$\circ$};
\draw (10,2) node {$\bullet$};
\draw (10,4) node {$\circ$};
\draw (10,6) node {$\circ$};
\draw (10,8) node {$\circ$};
\draw (11,1) node {$\bullet$};
\draw (11,3) node {$\circ$};
\draw (11,5) node {$\circ$};
\draw (11,7) node {$\circ$};
\draw (11,9) node {$\circ$};
\draw (12,2) node {$\bullet$};
\draw (12,4) node {$\circ$};
\draw (12,6) node {$\circ$};
\draw (12,8) node {$\circ$};
\draw (13,1) node {$\bullet$};
\draw (13,3) node {$\circ$};
\draw (13,5) node {$\circ$};
\draw (13,7) node {$\circ$};
\draw (13,9) node {$\circ$};
\end{tikzpicture}
&
\begin{tikzpicture}[scale=1/6,rotate=45,yscale=-1]
\draw (1,1) node {$\bullet$};
\draw (1,3) node {$\bullet$};
\draw (1,5) node {$\bullet$};
\draw (1,7) node {$\bullet$};
\draw (2,2) node {$\bullet$};
\draw (2,4) node {$\bullet$};
\draw (2,6) node {$\bullet$};
\draw (2,8) node {$\bullet$};
\draw (3,1) node {$\bullet$};
\draw (3,3) node {$\circ$};
\draw (3,5) node {$\circ$};
\draw (3,7) node {$\circ$};
\draw (4,2) node {$\bullet$};
\draw (4,4) node {$\circ$};
\draw (4,6) node {$\circ$};
\draw (4,8) node {$\circ$};
\draw (5,1) node {$\bullet$};
\draw (5,3) node {$\circ$};
\draw (5,5) node {$\circ$};
\draw (5,7) node {$\circ$};
\draw (6,2) node {$\bullet$};
\draw (6,4) node {$\circ$};
\draw (6,6) node {$\circ$};
\draw (6,8) node {$\circ$};
\draw (7,1) node {$\bullet$};
\draw (7,3) node {$\circ$};
\draw (7,5) node {$\circ$};
\draw (7,7) node {$\circ$};
\draw (8,2) node {$\bullet$};
\draw (8,4) node {$\circ$};
\draw (8,6) node {$\circ$};
\draw (8,8) node {$\circ$};
\draw (9,1) node {$\bullet$};
\draw (9,3) node {$\circ$};
\draw (9,5) node {$\circ$};
\draw (9,7) node {$\circ$};
\draw (10,2) node {$\bullet$};
\draw (10,4) node {$\circ$};
\draw (10,6) node {$\circ$};
\draw (10,8) node {$\circ$};
\draw (11,1) node {$\bullet$};
\draw (11,3) node {$\circ$};
\draw (11,5) node {$\circ$};
\draw (11,7) node {$\circ$};
\draw (12,2) node {$\bullet$};
\draw (12,4) node {$\circ$};
\draw (12,6) node {$\circ$};
\draw (12,8) node {$\circ$};
\draw (13,1) node {$\bullet$};
\draw (13,3) node {$\circ$};
\draw (13,5) node {$\circ$};
\draw (13,7) node {$\circ$};
\end{tikzpicture}
&
\begin{tikzpicture}[scale=1/6,rotate=45,yscale=-1]
\draw (1,3) node {$\bullet$};
\draw (1,5) node {$\bullet$};
\draw (1,7) node {$\bullet$};
\draw (1,9) node {$\bullet$};
\draw (2,2) node {$\bullet$};
\draw (2,4) node {$\bullet$};
\draw (2,6) node {$\bullet$};
\draw (2,8) node {$\bullet$};
\draw (3,3) node {$\bullet$};
\draw (3,5) node {$\circ$};
\draw (3,7) node {$\circ$};
\draw (3,9) node {$\circ$};
\draw (4,2) node {$\bullet$};
\draw (4,4) node {$\circ$};
\draw (4,6) node {$\circ$};
\draw (4,8) node {$\circ$};
\draw (5,3) node {$\bullet$};
\draw (5,5) node {$\circ$};
\draw (5,7) node {$\circ$};
\draw (5,9) node {$\circ$};
\draw (6,2) node {$\bullet$};
\draw (6,4) node {$\circ$};
\draw (6,6) node {$\circ$};
\draw (6,8) node {$\circ$};
\draw (7,3) node {$\bullet$};
\draw (7,5) node {$\circ$};
\draw (7,7) node {$\circ$};
\draw (7,9) node {$\circ$};
\draw (8,2) node {$\bullet$};
\draw (8,4) node {$\circ$};
\draw (8,6) node {$\circ$};
\draw (8,8) node {$\circ$};
\draw (9,3) node {$\bullet$};
\draw (9,5) node {$\circ$};
\draw (9,7) node {$\circ$};
\draw (9,9) node {$\circ$};
\draw (10,2) node {$\bullet$};
\draw (10,4) node {$\circ$};
\draw (10,6) node {$\circ$};
\draw (10,8) node {$\circ$};
\draw (11,3) node {$\bullet$};
\draw (11,5) node {$\circ$};
\draw (11,7) node {$\circ$};
\draw (11,9) node {$\circ$};
\draw (12,2) node {$\bullet$};
\draw (12,4) node {$\circ$};
\draw (12,6) node {$\circ$};
\draw (12,8) node {$\circ$};
\draw (13,3) node {$\bullet$};
\draw (13,5) node {$\circ$};
\draw (13,7) node {$\circ$};
\draw (13,9) node {$\circ$};
\end{tikzpicture}
\begin{tikzpicture}[scale=1/6,rotate=45,yscale=-1]
\draw (1,3) node {$\bullet$};
\draw (1,5) node {$\bullet$};
\draw (1,7) node {$\bullet$};
\draw (2,2) node {$\bullet$};
\draw (2,4) node {$\bullet$};
\draw (2,6) node {$\bullet$};
\draw (2,8) node {$\bullet$};
\draw (3,3) node {$\bullet$};
\draw (3,5) node {$\circ$};
\draw (3,7) node {$\circ$};
\draw (4,2) node {$\bullet$};
\draw (4,4) node {$\circ$};
\draw (4,6) node {$\circ$};
\draw (4,8) node {$\circ$};
\draw (5,3) node {$\bullet$};
\draw (5,5) node {$\circ$};
\draw (5,7) node {$\circ$};
\draw (6,2) node {$\bullet$};
\draw (6,4) node {$\circ$};
\draw (6,6) node {$\circ$};
\draw (6,8) node {$\circ$};
\draw (7,3) node {$\bullet$};
\draw (7,5) node {$\circ$};
\draw (7,7) node {$\circ$};
\draw (8,2) node {$\bullet$};
\draw (8,4) node {$\circ$};
\draw (8,6) node {$\circ$};
\draw (8,8) node {$\circ$};
\draw (9,3) node {$\bullet$};
\draw (9,5) node {$\circ$};
\draw (9,7) node {$\circ$};
\draw (10,2) node {$\bullet$};
\draw (10,4) node {$\circ$};
\draw (10,6) node {$\circ$};
\draw (10,8) node {$\circ$};
\draw (11,3) node {$\bullet$};
\draw (11,5) node {$\circ$};
\draw (11,7) node {$\circ$};
\draw (12,2) node {$\bullet$};
\draw (12,4) node {$\circ$};
\draw (12,6) node {$\circ$};
\draw (12,8) node {$\circ$};
\draw (13,3) node {$\bullet$};
\draw (13,5) node {$\circ$};
\draw (13,7) node {$\circ$};
\end{tikzpicture}
&
\\
\begin{tikzpicture}[scale=1/6,rotate=45,yscale=-1]
\draw (1,1) node {$\bullet$};
\draw (1,3) node {$\bullet$};
\draw (1,5) node {$\bullet$};
\draw (1,7) node {$\bullet$};
\draw (1,9) node {$\bullet$};
\draw (2,2) node {$\bullet$};
\draw (2,4) node {$\bullet$};
\draw (2,6) node {$\bullet$};
\draw (2,8) node {$\bullet$};
\draw (3,1) node {$\bullet$};
\draw (3,3) node {$\circ$};
\draw (3,5) node {$\circ$};
\draw (3,7) node {$\circ$};
\draw (3,9) node {$\circ$};
\draw (4,2) node {$\bullet$};
\draw (4,4) node {$\circ$};
\draw (4,6) node {$\circ$};
\draw (4,8) node {$\circ$};
\draw (5,1) node {$\bullet$};
\draw (5,3) node {$\circ$};
\draw (5,5) node {$\circ$};
\draw (5,7) node {$\circ$};
\draw (5,9) node {$\circ$};
\draw (6,2) node {$\bullet$};
\draw (6,4) node {$\circ$};
\draw (6,6) node {$\circ$};
\draw (6,8) node {$\circ$};
\draw (7,1) node {$\bullet$};
\draw (7,3) node {$\circ$};
\draw (7,5) node {$\circ$};
\draw (7,7) node {$\circ$};
\draw (7,9) node {$\circ$};
\draw (8,2) node {$\bullet$};
\draw (8,4) node {$\circ$};
\draw (8,6) node {$\circ$};
\draw (8,8) node {$\circ$};
\draw (9,1) node {$\bullet$};
\draw (9,3) node {$\circ$};
\draw (9,5) node {$\circ$};
\draw (9,7) node {$\circ$};
\draw (9,9) node {$\circ$};
\draw (10,2) node {$\bullet$};
\draw (10,4) node {$\circ$};
\draw (10,6) node {$\circ$};
\draw (10,8) node {$\circ$};
\draw (11,1) node {$\bullet$};
\draw (11,3) node {$\circ$};
\draw (11,5) node {$\circ$};
\draw (11,7) node {$\circ$};
\draw (11,9) node {$\circ$};
\draw (12,2) node {$\bullet$};
\draw (12,4) node {$\circ$};
\draw (12,6) node {$\circ$};
\draw (12,8) node {$\circ$};
\draw (13,1) node {$\bullet$};
\draw (13,3) node {$\circ$};
\draw (13,5) node {$\circ$};
\draw (13,7) node {$\circ$};
\draw (13,9) node {$\circ$};
\draw (14,2) node {$\bullet$};
\draw (14,4) node {$\circ$};
\draw (14,6) node {$\circ$};
\draw (14,8) node {$\circ$};
\end{tikzpicture}
&
\begin{tikzpicture}[scale=1/6,rotate=45,yscale=-1]
\draw (1,1) node {$\bullet$};
\draw (1,3) node {$\bullet$};
\draw (1,5) node {$\bullet$};
\draw (1,7) node {$\bullet$};
\draw (2,2) node {$\bullet$};
\draw (2,4) node {$\bullet$};
\draw (2,6) node {$\bullet$};
\draw (2,8) node {$\bullet$};
\draw (3,1) node {$\bullet$};
\draw (3,3) node {$\circ$};
\draw (3,5) node {$\circ$};
\draw (3,7) node {$\circ$};
\draw (4,2) node {$\bullet$};
\draw (4,4) node {$\circ$};
\draw (4,6) node {$\circ$};
\draw (4,8) node {$\circ$};
\draw (5,1) node {$\bullet$};
\draw (5,3) node {$\circ$};
\draw (5,5) node {$\circ$};
\draw (5,7) node {$\circ$};
\draw (6,2) node {$\bullet$};
\draw (6,4) node {$\circ$};
\draw (6,6) node {$\circ$};
\draw (6,8) node {$\circ$};
\draw (7,1) node {$\bullet$};
\draw (7,3) node {$\circ$};
\draw (7,5) node {$\circ$};
\draw (7,7) node {$\circ$};
\draw (8,2) node {$\bullet$};
\draw (8,4) node {$\circ$};
\draw (8,6) node {$\circ$};
\draw (8,8) node {$\circ$};
\draw (9,1) node {$\bullet$};
\draw (9,3) node {$\circ$};
\draw (9,5) node {$\circ$};
\draw (9,7) node {$\circ$};
\draw (10,2) node {$\bullet$};
\draw (10,4) node {$\circ$};
\draw (10,6) node {$\circ$};
\draw (10,8) node {$\circ$};
\draw (11,1) node {$\bullet$};
\draw (11,3) node {$\circ$};
\draw (11,5) node {$\circ$};
\draw (11,7) node {$\circ$};
\draw (12,2) node {$\bullet$};
\draw (12,4) node {$\circ$};
\draw (12,6) node {$\circ$};
\draw (12,8) node {$\circ$};
\draw (13,1) node {$\bullet$};
\draw (13,3) node {$\circ$};
\draw (13,5) node {$\circ$};
\draw (13,7) node {$\circ$};
\draw (14,2) node {$\bullet$};
\draw (14,4) node {$\circ$};
\draw (14,6) node {$\circ$};
\draw (14,8) node {$\circ$};
\end{tikzpicture}
&
\begin{tikzpicture}[scale=1/6,rotate=45,yscale=-1]
\draw (1,3) node {$\bullet$};
\draw (1,5) node {$\bullet$};
\draw (1,7) node {$\bullet$};
\draw (1,9) node {$\bullet$};
\draw (2,2) node {$\bullet$};
\draw (2,4) node {$\bullet$};
\draw (2,6) node {$\bullet$};
\draw (2,8) node {$\bullet$};
\draw (3,3) node {$\bullet$};
\draw (3,5) node {$\circ$};
\draw (3,7) node {$\circ$};
\draw (3,9) node {$\circ$};
\draw (4,2) node {$\bullet$};
\draw (4,4) node {$\circ$};
\draw (4,6) node {$\circ$};
\draw (4,8) node {$\circ$};
\draw (5,3) node {$\bullet$};
\draw (5,5) node {$\circ$};
\draw (5,7) node {$\circ$};
\draw (5,9) node {$\circ$};
\draw (6,2) node {$\bullet$};
\draw (6,4) node {$\circ$};
\draw (6,6) node {$\circ$};
\draw (6,8) node {$\circ$};
\draw (7,3) node {$\bullet$};
\draw (7,5) node {$\circ$};
\draw (7,7) node {$\circ$};
\draw (7,9) node {$\circ$};
\draw (8,2) node {$\bullet$};
\draw (8,4) node {$\circ$};
\draw (8,6) node {$\circ$};
\draw (8,8) node {$\circ$};
\draw (9,3) node {$\bullet$};
\draw (9,5) node {$\circ$};
\draw (9,7) node {$\circ$};
\draw (9,9) node {$\circ$};
\draw (10,2) node {$\bullet$};
\draw (10,4) node {$\circ$};
\draw (10,6) node {$\circ$};
\draw (10,8) node {$\circ$};
\draw (11,3) node {$\bullet$};
\draw (11,5) node {$\circ$};
\draw (11,7) node {$\circ$};
\draw (11,9) node {$\circ$};
\draw (12,2) node {$\bullet$};
\draw (12,4) node {$\circ$};
\draw (12,6) node {$\circ$};
\draw (12,8) node {$\circ$};
\draw (13,3) node {$\bullet$};
\draw (13,5) node {$\circ$};
\draw (13,7) node {$\circ$};
\draw (13,9) node {$\circ$};
\draw (14,2) node {$\bullet$};
\draw (14,4) node {$\circ$};
\draw (14,6) node {$\circ$};
\draw (14,8) node {$\circ$};
\end{tikzpicture}
\begin{tikzpicture}[scale=1/6,rotate=45,yscale=-1]
\draw (1,3) node {$\bullet$};
\draw (1,5) node {$\bullet$};
\draw (1,7) node {$\bullet$};
\draw (2,2) node {$\bullet$};
\draw (2,4) node {$\bullet$};
\draw (2,6) node {$\bullet$};
\draw (2,8) node {$\bullet$};
\draw (3,3) node {$\bullet$};
\draw (3,5) node {$\circ$};
\draw (3,7) node {$\circ$};
\draw (4,2) node {$\bullet$};
\draw (4,4) node {$\circ$};
\draw (4,6) node {$\circ$};
\draw (4,8) node {$\circ$};
\draw (5,3) node {$\bullet$};
\draw (5,5) node {$\circ$};
\draw (5,7) node {$\circ$};
\draw (6,2) node {$\bullet$};
\draw (6,4) node {$\circ$};
\draw (6,6) node {$\circ$};
\draw (6,8) node {$\circ$};
\draw (7,3) node {$\bullet$};
\draw (7,5) node {$\circ$};
\draw (7,7) node {$\circ$};
\draw (8,2) node {$\bullet$};
\draw (8,4) node {$\circ$};
\draw (8,6) node {$\circ$};
\draw (8,8) node {$\circ$};
\draw (9,3) node {$\bullet$};
\draw (9,5) node {$\circ$};
\draw (9,7) node {$\circ$};
\draw (10,2) node {$\bullet$};
\draw (10,4) node {$\circ$};
\draw (10,6) node {$\circ$};
\draw (10,8) node {$\circ$};
\draw (11,3) node {$\bullet$};
\draw (11,5) node {$\circ$};
\draw (11,7) node {$\circ$};
\draw (12,2) node {$\bullet$};
\draw (12,4) node {$\circ$};
\draw (12,6) node {$\circ$};
\draw (12,8) node {$\circ$};
\draw (13,3) node {$\bullet$};
\draw (13,5) node {$\circ$};
\draw (13,7) node {$\circ$};
\draw (14,2) node {$\bullet$};
\draw (14,4) node {$\circ$};
\draw (14,6) node {$\circ$};
\draw (14,8) node {$\circ$};
\end{tikzpicture}
&
\end{tabular}
\caption{The eight types of non-degenerate tilted rectangles $R_{[a,b],[c,d]}$ depending on the relative parities of $a,b,c,d$.  The west boundaries are displayed in black.}
\label{f.tilted}
\end{figure}

\fref{tilted} displays several examples of tilted rectangles for which one can easily verify the algorithm from the proof of \lref{extension}.  By quantifying the rate of growth in this algorithm, we obtain the following.

\begin{lemma}
\label{l.exponential}
If $H \psi = \lambda \psi$ in $R = R_{[2,a-1],[2,b-1]}$, then
\begin{equation*}
\| \psi \|_{\ell^\infty(R_{[1,a],[1,b]})} \leq e^{C b \log a} \| \psi \|_{\ell^\infty(\partial^w R_{[1,a],[1,b]})}.
\end{equation*}
\end{lemma}

\begin{proof}
We show that, if $H \psi = \lambda \psi$ in $R_{[2,a-1],[2,b-1]}$ and $|\psi| \leq 1$ in $\partial^w R_{[1,a],[1,b]}$, then $|\psi_{s,t}| \leq (\alpha s)^t$ for $(s,t) \in R_{[1,a],[1,b]}$.  Here $\alpha \geq 1$ is a universal constant to be determined.  We prove this by induction on $(s,t)$.  First, we note that, if $(s,t) \in \partial^w R_{[1,a],[1,b]}$, then $|\psi_{s,t}| \leq 1 \leq (\alpha s)^t$ holds by assumption.  Second, we note that, if $(s, t) \in R_{[3,a],[3,b]}$, then we can use the equation \eref{tilted2} and the induction hypothesis to estimate
\begin{equation*}
\begin{aligned}
|\psi_{s,t}|
& \leq C |\psi_{s-1,t-1}| + |\psi_{s-2,t}| + |\psi_{s,t-2}| + |\psi_{s-2,t-2}| \\
& \leq C (\alpha (s-1))^{t-1} + (\alpha (s-2))^t + (\alpha s)^{t-2} + (\alpha(s-2))^{t-2} \\
& \leq (\alpha s)^t ((\tfrac{s-2}{s})^t + C (\alpha s)^{-1}) \\
& \leq (\alpha s)^t (1 - 2 s^{-1} + C (\alpha s)^{-1}) \\
& \leq (\alpha s)^{t}.
\end{aligned}
\end{equation*}
Here we used $|4 + V_{s-1,t-1} - \lambda| \leq C$ and $\alpha \geq 2 C$.
\end{proof}

Differentiating the expression for $\psi$ on $R$ in terms of $\psi$ on $\partial^w R$ with respect to $\lambda$ and using \lref{exponential}, we obtain the following quantitative estimate of the dependence on $\lambda$.

\begin{lemma}
\label{l.lambdadependence}
If $H \psi_0 = \lambda_0 \psi_0$ and $H \psi_1 = \lambda_1 \psi_1$  in $R_{[2,a-1],[2,b-1]}$ and $\psi_0 = \psi_1$ in $\partial^w R_{[1,a],[1,b]}$, then
\begin{equation*}
\| \psi_0 - \psi_1 \|_{\ell^\infty(R_{[1,a],[1,b]})} \leq e^{C a \log b} \| \psi_0 \|_{\ell^\infty(\partial^w R_{[1,a],[1,b]})} |\lambda_0 - \lambda_1|.
\end{equation*}
\end{lemma}

\begin{proof}
Using \lref{extension}, let $\psi_\lambda : R_{[1,a],[1,b]} \to \R$ be the unique function such that $H \psi_\lambda = \lambda \psi_\lambda$ in $R_{[2,a-1],[2,b-1]}$ and $\psi_\lambda = \psi_0$ in $\partial^w R_{[1,a],[1,b]}$.  Since the derivative $\dot \psi_\lambda = \frac{d}{d\lambda} \psi_\lambda$ satisfies
$$\begin{cases}
H \dot \psi_\lambda = \lambda \dot \psi_\lambda + \psi_\lambda & \mbox{in } R_{[3,a],[3,b]} \\
\dot \psi_\lambda = 0 & \mbox{in } R_{[1,a],[1,b]} \setminus R_{[3,a],[3,b]},
\end{cases}$$
the proof of \lref{exponential} gives
\begin{equation*}
\| \dot \psi_\lambda \|_{\ell^\infty(R_{[1,a],[1,b]})} \leq e^{C b\log a} \|\psi_0\|_{\ell^\infty(\partial^w R_{[1,a],[1,b]})}.
\end{equation*}
Integrating over $\lambda \in [\lambda_0,\lambda_1]$ gives the desired estimate.
\end{proof}

\subsection{Key lemma}

We recall a key ingredient \cite{Buhovsky-Logunov-Malinnikova-Sodin}*{Lemma 3.4} used in the upper bound in the Liouville theorem for harmonic functions on the lattice. 

\begin{lemma}
If $a \geq C b$, $\Delta u = 0$ in $R_{[2,a-1],[2,b-1]}$, $|u| \leq 1$ in $R_{[1,a],[1,2]}$, and $|u| \leq 1$ in a $1/2$ fraction of $R_{[1,a],[b,b]}$, then $|u| \leq e^{C b \log a}$ in $R_{[1,a],[1,b]}$. \qed
\end{lemma}

Note that the bounds on $u$ in the above lemma are on opposite sides of the rectangle.  The lemma says that, if $u$ is harmonic in a rectangle, bounded on the northwest boundary, and bounded on half of the southeast boundary, then $u$ is bounded on the entire rectangle.

The main idea of the proof of the above lemma is the observation that, if $u = 0$ on $R_{[1,a],[1,2]}$ and $u$ is harmonic in $R_{[1,a],[1,b]}$, then function $v_{s,t} = (-1)^{(s+t)/2} u_{s,t}$ is a polynomial of degree at most $t - 2$ in the variable $s$.  Using the polynomial structure, the Remez inequality \cite{Bojanov} on the southeast boundary $R_{[1,a],[b,b]}$ takes us from bounded on half of the points to bounded on all of the points.  This argument is delicate, and appears to break down in the presence of a potential.  Indeed, there is no reason to expect that a solution of $H \psi = \lambda \psi$ with $\psi = 0$ on $R_{[1,a],[1,2]}$ will have polynomial structure.

Instead, we view the solution map $\R^{\partial^w R} \to \R^{R}$ given by \lref{extension} as a random linear operator.  We bound the right inverse of the solution map using an $\ep$-net in combination with a large deviations estimate.  Here we need that the rectangle $R_{[1,a],[1,b]}$ is extremely thin, requiring $a \geq C b^2 \log a$.

\begin{lemma}
\label{l.nice}
There are $\alpha > 1 > \ep > 0$ such that, if
\begin{enumerate}
\item $a \geq \alpha b^2 \log a \geq \alpha$
\item $F \subseteq \Z^2$ is $(\ep,-)$-sparse in $R_{[1,a],[1,b]}$
\item $v : F \to \{ 0, 1 \}$
\item $\mathcal E_{\mathrm{ni}}(R_{[1,a],[1,b]})$ denotes the event that
\begin{equation*}
\begin{cases}
|\lambda - \bar \lambda| \leq e^{- \alpha b \log a} \\
H \psi = \lambda \psi \mbox{ in } R_{[2,a-1],[2,b-1]} \\
|\psi| \leq 1 \mbox{ in } R_{[1,a],[1,2]} \\
|\psi| \leq 1 \mbox{ in a $1-\ep$ fraction of } R_{[1,a],[b-1,b]}
\end{cases}
\end{equation*}
implies $|\psi| \leq e^{\alpha b \log a}$ in $R_{[1,a],[1,b]}$,
\end{enumerate}
then $\P[\mathcal E_{\mathrm{ni}}(R_{[1,a],[1,b]}) | V_F = v] \geq 1 - e^{- \ep a}.$
\end{lemma}

\begin{proof}
Let $\mathcal E'_{\mathrm{ni}}(R_{[1,a],[1,b]})$ denote the event that
\begin{equation*}
\begin{cases}
H \psi = \bar \lambda \psi \mbox{ in } R_{[2,a-1],[2,b-1]} \\
\psi = 0 \mbox{ in } R_{[1,a],[1,2]} \\
\max_{R_{[1,2],[1,b]}} |\psi| \geq 1
\end{cases}
\end{equation*}
implies $|\psi| \geq e^{-\frac12 \alpha b \log a}$ in a $2\ep$ fraction of $R_{[1,a],[b-1,b]}$.

\begin{claim}
If $\alpha \geq C$, then $\mathcal E_{\mathrm{ni}} \supseteq \mathcal E'_{\mathrm{ni}}$.
\end{claim}

Suppose $|\lambda - \bar \lambda| \leq e^{-\alpha b \log a}$, $H \psi = \lambda \psi$ in $R_{[2,a-1],[2,b-1]}$, $|\psi| \leq 1$ in $R_{[1,a],[1,2]}$, and $|\psi| \leq 1$ in a $1-\ep$ fraction of $R_{[1,a],[b-1,b]}$.  By \lref{extension}, there are unique $\psi', \psi'' : R_{[1,a],[1,b]} \to \R$ such that
\begin{equation*}
\begin{cases}
H \psi' = \lambda \psi' & \mbox{in } R_{[2,a-1],[2,b-1]} \\
\psi' = 0 & \mbox{in } R_{[1,a],[1,2]} \\
\psi' = \psi & \mbox{in } R_{[1,2],[3,b]}
\end{cases}
\end{equation*}
and
\begin{equation*}
\begin{cases}
H \psi'' = \bar \lambda \psi'' & \mbox{in } R_{[2,a-1],[2,b-1]} \\
\psi'' = 0 & \mbox{in } R_{[1,a],[1,2]} \\
\psi'' = \psi & \mbox{in } R_{[1,2],[3,b]}.
\end{cases}
\end{equation*}
Assuming $\mathcal E'_{\mathrm{ni}}$ holds, we see that $|\psi''| \geq e^{- \frac12 \alpha b \log a} \max_{R_{[1,2],[3,b]}} |\psi|$ in a $2\ep$ fraction of $R_{[1,a],[b-1,b]}$.  Since \lref{lambdadependence} gives
\begin{equation*}
\max_{R_{[1,a],[1,b]}} |\psi' - \psi''| \leq e^{(C - \alpha) b \log a} \max_{R_{[1,2],[3,b]}} |\psi|,
\end{equation*}
we see that
\begin{equation*}
\begin{aligned}
|\psi'| & \geq |\psi''| - |\psi' - \psi''| \\
& \geq (e^{-\frac12 \alpha b \log a} - e^{(C - \alpha) b \log a}) \max_{R_{[1,2],[3,b]}} |\psi| \\
& \geq \tfrac12 e^{-\tfrac12 \alpha b \log a} \max_{R_{[1,2],[3,b]}} |\psi| \\
\end{aligned}
\end{equation*}
in a $2\ep$ fraction of $R_{[1,a],[b-1,b]}$.  Since \lref{exponential} gives\begin{equation*}
\max_{R_{[1,a],[1,b]}} |\psi' - \psi| \leq e^{C b \log a},
\end{equation*}
we obtain that $|\psi| \geq \tfrac12 e^{- \frac12 \alpha b \log a} \max_{R_{[1,2],[3,b]}} |\psi| - e^{C b \log a}$ in a $2 \ep$ fraction of $R_{[1,a],[b-1,b]}$.  Since $|\psi| \leq 1$ on a $1-\ep$ fraction of $R_{[1,a],[b-1,b]}$, it follows that $\max_{R_{[1,2],[3,b]}} |\psi| \leq e^{(C + \frac12 \alpha)b \log a}$.  Another application of \lref{exponential} gives $|\psi| \leq e^{(C + \frac12 \alpha) b \log a}$ in $R_{[1,a],[1,b]}$.  In particular, we see that $\mathcal E'_{\mathrm{ni}}$ implies $\mathcal E_{\mathrm{ni}}$.

We now estimate the probability that $\mathcal E'_{\mathrm{ni}}$ holds.  Recall that, if $H \psi = \bar \lambda \psi$ in $R_{[2,a-1],[2,b-1]}$ and $\psi = 0$ in $R_{[1,a],[1,2]}$, then the values of $\psi$ on the whole tilted rectangle $R_{[1,a],[1,b]}$ are determined by the potential $V$ and the values of $\psi$ on the southwest boundary $R_{[1,2],[3,b]}$.  Let us write $\psi^0$ and $\psi^1$ for the restriction of $\psi$ to $R_{[1,2],[3,b]}$ and $R_{[1,a],[b-1,b]}$, respectively.  We prove the lemma by studying the random linear mapping $\psi^0 \mapsto \psi^1$.  Our goal is to show that, with high probability, if $|\psi^0| \geq 1$ on at least one site, then $|\psi^1| \geq e^{- \alpha b \log a}$ on an $\ep$ fraction of its domain.  In particular, the lemma follows from \clref{inversebound} below.

\begin{claim}
For any $\psi^0 : R_{[1,2],[3,b]} \to \R$, there is $(s_0,t_0) \in R_{[1,2],[3,b]}$ such that
\begin{equation*}
|\psi| \geq e^{-C b \log a} \| \psi^0 \|_\infty \quad \mbox{in } R_{[1,a],[t_0,t_0]}
\end{equation*}
holds for all choices of potential $V$.  Note that $s_0 \in [1,2]$ records the parity of $t_0$.
\end{claim}

For $\beta \geq 1$ large universal to be determined, let $(s_0,t_0) \in R_{[1,2],[3,b]}$ maximize $a^{-\beta t_0} |\psi_{s_0,t_0}|$.  Note that it is enough to prove $|\psi| \geq 1/2$ on $R_{[1,a],[t_0,t_0]}$ under the assumption $\psi_{s_0,t_0} = 1.$ For $t \in [1,t_0]$, let $m_t = \| \psi \|_{\ell^\infty(R_{[1,a],[1,t]})}$.  Using \eref{tilted2}, we observe that, if $(s,t) \in R_{[3,a],[3,t_0]}$, then
\begin{equation*}
|\psi_{s,t}| \leq |\psi_{s-2,t}| + C m_{t-1} + C m_{t-2}.
\end{equation*}
Since $|\psi_{s,t}| \leq a^{\beta (t - t_0)}$ for $(s,t) \in R_{[1,2],[3,t_0]}$, induction gives
\begin{equation*}
m_t \leq a^{\beta (t-t_0)} + C a m_{t-1} \quad \mbox{for } t \geq 3.
\end{equation*}
Since $m_1 = m_2 = 0$, assuming $\beta \geq 1$ large gives $m_t \leq 2 a^{\beta (t - t_0)} |\psi_{s_0,t_0}|$.  Using \eref{tilted2} again, we observe that, if $(s,t_0) \in R_{[3,a],[t_0,t_0]}$, then
\begin{equation*}
|\psi_{s,t_0}| \geq |\psi_{s-2,t_0}| - C a^{-\beta} |\psi_{s_0,t_0}|.
\end{equation*}
By $\psi_{s_0,t_0} = 1$ and induction, we obtain $|\psi| \geq 1 - C a^{1-\beta}$ in $R_{[1,a],[t_0,t_0]}$.   Assuming $\beta \geq 1$ large gives $|\psi| \geq 1/2$ in $R_{[1,a],[t_0,t_0]}.$

\begin{figure}
\begin{tikzpicture}[scale=1/2,rotate=45,yscale=-1]
\draw (9,13) node[above] {$(s_1,t_1)$};
\draw (8,6) node[above] {$(s_1-1,t_0)$};
\draw (2,6) node[above] {$(s_0,t_0)$};
\draw (1,1) node {$\circ$};
\draw (1,3) node {$\circ$};
\draw (1,5) node {$\circ$};
\draw (1,7) node {$\circ$};
\draw (1,9) node {$\circ$};
\draw (1,11) node {$\circ$};
\draw (1,13) node {$\circ$};
\draw (2,2) node {$\circ$};
\draw (2,4) node {$\circ$};
\draw (2,6) node {$\bullet$};
\draw (2,8) node {$\circ$};
\draw (2,10) node {$\circ$};
\draw (2,12) node {$\circ$};
\draw (3,1) node {$\circ$};
\draw (3,3) node {$\circ$};
\draw (3,5) node {$\circ$};
\draw (3,7) node {$\circ$};
\draw (3,9) node {$\circ$};
\draw (3,11) node {$\circ$};
\draw (3,13) node {$\circ$};
\draw (4,2) node {$\circ$};
\draw (4,4) node {$\circ$};
\draw (4,6) node {$\circ$};
\draw (4,8) node {$\circ$};
\draw (4,10) node {$\circ$};
\draw (4,12) node {$\circ$};
\draw (5,1) node {$\circ$};
\draw (5,3) node {$\circ$};
\draw (5,5) node {$\circ$};
\draw (5,7) node {$\circ$};
\draw (5,9) node {$\circ$};
\draw (5,11) node {$\circ$};
\draw (5,13) node {$\circ$};
\draw (6,2) node {$\circ$};
\draw (6,4) node {$\circ$};
\draw (6,6) node {$\circ$};
\draw (6,8) node {$\circ$};
\draw (6,10) node {$\circ$};
\draw (6,12) node {$\circ$};
\draw (7,1) node {$\bullet$};
\draw (7,3) node {$\circ$};
\draw (7,5) node {$\circ$};
\draw (7,7) node {$\circ$};
\draw (7,9) node {$\circ$};
\draw (7,11) node {$\circ$};
\draw (7,13) node {$\bullet$};
\draw (8,2) node {$\bullet$};
\draw (8,4) node {$\bullet$};
\draw (8,6) node {$\bullet$};
\draw (8,8) node {$\bullet$};
\draw (8,10) node {$\bullet$};
\draw (8,12) node {$\bullet$};
\draw (9,1) node {$\bullet$};
\draw (9,3) node {$\circ$};
\draw (9,5) node {$\circ$};
\draw (9,7) node {$\circ$};
\draw (9,9) node {$\circ$};
\draw (9,11) node {$\circ$};
\draw (9,13) node {$\bullet$};
\draw (10,2) node {$\circ$};
\draw (10,4) node {$\circ$};
\draw (10,6) node {$\circ$};
\draw (10,8) node {$\circ$};
\draw (10,10) node {$\circ$};
\draw (10,12) node {$\circ$};
\draw (11,1) node {$\circ$};
\draw (11,3) node {$\circ$};
\draw (11,5) node {$\circ$};
\draw (11,7) node {$\circ$};
\draw (11,9) node {$\circ$};
\draw (11,11) node {$\circ$};
\draw (11,13) node {$\circ$};
\draw (12,2) node {$\circ$};
\draw (12,4) node {$\circ$};
\draw (12,6) node {$\circ$};
\draw (12,8) node {$\circ$};
\draw (12,10) node {$\circ$};
\draw (12,12) node {$\circ$};
\draw (13,1) node {$\circ$};
\draw (13,3) node {$\circ$};
\draw (13,5) node {$\circ$};
\draw (13,7) node {$\circ$};
\draw (13,9) node {$\circ$};
\draw (13,11) node {$\circ$};
\draw (13,13) node {$\circ$};
\draw (14,2) node {$\circ$};
\draw (14,4) node {$\circ$};
\draw (14,6) node {$\circ$};
\draw (14,8) node {$\circ$};
\draw (14,10) node {$\circ$};
\draw (14,12) node {$\circ$};
\draw (15,1) node {$\circ$};
\draw (15,3) node {$\circ$};
\draw (15,5) node {$\circ$};
\draw (15,7) node {$\circ$};
\draw (15,9) node {$\circ$};
\draw (15,11) node {$\circ$};
\draw (15,13) node {$\circ$};
\end{tikzpicture}
\caption{A schematic for the proof of \clref{martingale}}
\label{f.martingale}
\end{figure}

\begin{claim}
\label{cl.martingale}
For any fixed $\psi^0 : R_{[1,2],[1,b]} \to \R$,
\begin{equation*}
\P[|\{|\psi^1| \geq e^{- C b \log a} \| \psi^0 \|_\infty \}| \geq \ep a | V_F = v] \geq 1 - e^{-c a}.
\end{equation*}
\end{claim}

Select $(s_0,t_0) \in R_{[1,2],[3,b]}$ as in the previous claim.  Suppose for the moment that $(s_1,t_1) \in R_{[3,a],[b-1,b]}$, and $(s_1-1,t_0) \in R_{[2,a],[t_0,t_0]} \setminus F$.  Forming the alternating sum of the equation \eref{tilted1} at the points $(s,t_1-1) \in R_{[3,a],[1,s_1]}$ and using $\psi = 0$ on $R_{[1,a],[1,2]}$, we obtain
\begin{equation*}
\psi_{s_1,t_1} = - \psi_{s_1-2,t_1} + \sum_{0 \leq k \leq \frac{t_1-1}{2}} (-1)^k (4 - \bar \lambda + V_{s_1-1,t_1-1-2k}) \psi_{s_1-1,t_1-1-2k}.
\end{equation*}
See \fref{martingale} for a schematic of this computation.

Since the values $\psi_{s_1-1,t_1 - 1 - 2k}$ depend only on $\psi^0$ and the potential $V$ on the tilted rectangle $R_{[1,s_1-2],[1,b]}$, we see that $\psi_{s_1,t_1}$ depends on $V_{s_1-1,t_0}$ only through the term
\begin{equation*}
(-1)^{t_1-t_0-1} (4 - \bar \lambda + V_{s_1-1,t_0}) \psi_{s_1-1,t_0}.
\end{equation*}
Since $|\psi_{s_1-1,t_0}| \geq e^{-C b \log a} \|\psi^0\|_\infty$ holds almost surely, we obtain
\begin{equation*}
\P[|\psi_{s_1,t_1}| \geq e^{-C b \log a} \| \psi^0 \|_\infty | V_{F \cup R_{[1,s_1-2],[1,b]}}] \geq 1/2.
\end{equation*}
That is, $\psi_{s_1,t_1}$ is sensitive to the variation of $V_{s_1-1,t_0}$.

Now, let $s_1, ..., s_K \in [1,a]$ be an increasing list of all the integers $s_k \in [1,a]$ such that $(s_k,t_1) \in R_{[3,a],[b-1,b]}$ such that $(s_k-1,t_0) \in R_{[2,a],[t_0,t_0]} \setminus F$.  Since $F$ is $(\ep,-)$-sparse, there are at least $K \geq (\tfrac12 - \ep) a - 3 \geq c a$ such integers.  Let $\mathcal F_k$ denote the $\sigma$-algebra generated by $V_{F \cup R_{[1,s_k-1],[1,b]}}$.  By the above, we see that $\psi_{s_k,t_1}$ is $\mathcal F_k$-measurable and that
\begin{equation*}
\P[|\psi_{s_k,t_1}| \geq e^{-C b \log a} \| \psi^0 \|_\infty | \mathcal F_{k-1}] \geq 1/2.
\end{equation*}
Applying Azuma's inequality to the sum $\sum_k \id_{|\psi_{s_k,t_1}| \geq e^{-C b \log a} \| \psi^0 \|_\infty}$ of indicator functions, we see that $$\P[\#\{ k : |\psi_{s_k,t_1}| \geq e^{-C b \log a} \| \psi^0 \|_\infty \} \leq \ep n] \leq e^{- c (c - \ep) a}$$
Therefore, we conclude by assuming $\ep > 0$ is small.

\begin{claim}
\label{cl.inversebound}
If $X$ denotes the space of $\psi^0 : R_{[1,2],[3,b]} \to \R$ with $\| \psi^0 \|_\infty = 1$, then
\begin{equation*}
\P[\inf_{\psi^0 \in X} |\{|\psi^1| \geq e^{- C b \log a} \}| \geq \ep a | V_F = v] \geq 1 - e^{-\ep a}.
\end{equation*}
\end{claim}

For $\beta \geq 1$ large universal to be determined, we can choose a finite subset $\tilde X \subseteq X$ such that $|\tilde X| \leq e^{C \beta b^2 \log a}$ and, for any $\psi^0 \in X$, there is a $\tilde \psi^0 \in \tilde X$ with $\|\psi^0 - \tilde \psi^0\|_\infty \leq e^{-2 \beta b \log a}$.  By \lref{exponential}, $\| \psi^0 - \tilde \psi^0 \|_\infty \leq e^{- 2 \beta b \log a}$ implies $\| \psi^1 -  \tilde \psi^1 \|_\infty \leq e^{(C - 2 \beta) b \log a}$.  In particular, we have
\begin{equation*}
\inf_{\psi^0 \in X} |\{|\psi^1| \geq e^{- \beta b \log a} - e^{(C - 2\beta) b \log a} \}| \geq \min_{\tilde \psi^0 \in \tilde X} |\{ |\tilde \psi^1| \geq e^{- \beta b \log a} \}|
\end{equation*}
By the previous claim and a union bound,
\begin{equation*}
\P[\min_{\tilde \psi^0 \in \tilde X} |\{ |\tilde \psi^1| \geq e^{- \beta b \log a} \}| \geq \ep a | V_F = v] \geq 1 - e^{C \beta b^2 \log a - c a},
\end{equation*}
holds provided $\beta \geq 1$ sufficiently large.  Assuming $a \geq C \beta b^2 \log a$, we obtain the claim.
\end{proof}

\subsection{Growth lemma}

Using our key lemma, we prove a ``growth lemma'' suitable for use in a Calderon--Zygmund stopping time argument.  Our proof is similar to that of \cite{Buhovsky-Logunov-Malinnikova-Sodin}*{Lemma 3.6} except that we are forced to use extremely thin rectangles.  This leads to large support on only $\ell(Q)^{3/2-\ep}$ many points.

\begin{lemma}
\label{l.excellent}
For every small $\ep > 0$, there is a large $\alpha > 1$ such that, if
\begin{enumerate}
\item $Q$ tilted square with $\ell(Q) \geq \alpha$
\item $F \subseteq \Z^2$ is $\ep$-sparse in $2Q$
\item $v : F \to \{ 0, 1 \}$
\item $\mathcal E_{\mathrm{ex}}(Q,F)$ denotes the event that
\begin{equation*}
\begin{cases}
|\lambda - \bar \lambda| \leq e^{-\alpha (\ell(Q) \log \ell(Q))^{1/2}} \\
H \psi = \lambda \psi \mbox{ in } 2Q \\
|\psi| \leq 1 \mbox{ in } \tfrac12 Q \\
|\psi| \leq 1 \mbox{ in a $1-\ep (\ell(Q) \log \ell(Q))^{-1/2}$ fraction of $2Q \setminus F$}
\end{cases}
\end{equation*}
implies $|\psi| \leq e^{\alpha \ell(Q) \log \ell(Q)}$ in $Q$.
\end{enumerate}
then $\P[\mathcal E_{\mathrm{ex}}(Q,F) | V_F = v] \geq 1 - e^{-\ep \ell(Q)}.$
\end{lemma}

\begin{proof}
For some large $a \geq C$, let $\mathcal E_{\mathrm{ex}}'$ denote the event that
\begin{equation}
\label{e.exprimeh}
\begin{cases}
|\lambda - \bar \lambda| \leq e^{-\alpha (a \log a)^{1/2}} \\
H \psi = \lambda \psi \mbox{ in } 4 R_{[1,a],[1,a]} \\
|\psi| \leq 1 \mbox{ in } R_{[1,a],[1,a]} \\
|\psi| \leq 1 \mbox{ in a $(1-\ep (a \log a)^{-1/2})$ fraction of $4 R_{[1,a],[1,a]} \setminus F$}
\end{cases}
\end{equation}
implies $|\psi| \leq e^{\alpha a \log a}$ in $R_{[1,a],[1,2a]}$.  By the $90^\circ$ symmetry of our problem and a covering argument, it is enough to prove that, for every small $\ep > 0$, there is a large $\alpha > 1$ such that $\P[\mathcal E_{\mathrm{ex}}' | V_F = v] \geq 1 - e^{-\ep a}$.

\begin{figure}
\begin{tikzpicture}[scale=1/8,rotate=45]
\draw (-16,-16) node[below] {$4 R_{[1,a],[1,a]}$} rectangle (16,16);
\draw (-4,-4) node[below] {$R_{[1,a],[1,a]}$} rectangle (4,4);
\draw (4,-4) rectangle (12,4);
\foreach \x in {4,5,...,12} {
\draw (\x,-4) -- (\x,4);
}
\draw[->,>=latex,dotted] (9.5,-30) node[below] {$R_{[1,a],[b_k-1,b_{k+1}]}$} -- (9.5,-4);
\end{tikzpicture}
\caption{A schematic for the proof of \lref{excellent}}
\label{f.excellent}
\end{figure}

Let $\alpha' > 1 > \ep' > 0$ denote a valid pair of constants from \lref{nice}.  By a union bound, the event\begin{equation*}
\mathcal E_{\mathrm{ni}} = \bigcap_{\substack{[c,d] \subseteq [1,\frac52 a] \\ \alpha (d - c + 1)^2 \log a \leq a}} \mathcal E_{\mathrm{ni}}(R_{[1,a],[c,d]})
\end{equation*}
satisfies $\P[\mathcal E_{\mathrm{ni}} | V_F = v] \geq 1 - e^{-\ep' a + C \log a}$.  It suffices to prove that, for all $\ep \in (0,c \ep' (\alpha')^{-1/2})$, there is a large $\alpha > \alpha'$ such that $\mathcal E_{\mathrm{ex}}' \supseteq \mathcal E_{\mathrm{ni}}$.  Henceforth we assume that $\mathcal E_{\mathrm{ni}}$ and \eref{exprimeh} hold.  Our goal is to show that $|\psi| \leq e^{\alpha a \log a}$ in $R_{[1,a],[1,2a]}$.

\begin{claim}
There is a sequence $b_0 \leq \cdots \leq b_K \in [a,\tfrac52 a]$ such that
\begin{enumerate}
\item $b_0 = a$
\item $b_K \geq 2 a$
\item $\tfrac12 a \leq \alpha' (b_{k+1} - b_k + 2)^2 \log a \leq a$ for $0 \leq k < K$
\item $|\psi| \leq 1$ on a $1- \ep'$ fraction of $R_{[1,a],[b_{k+1}-1,b_{k+1}]}$ for $0 \leq k < K$
\end{enumerate}
\end{claim}

Let $B$ denote the largest even integer such that $\alpha' B^2 \log a \leq a$. Assume that $b_k$ is already defined.  Decompose the rectangle $R_{[1,a],[b_k+B/2,b_k+B]}$ as a disjoint union of diagonals,
\begin{equation*}
R_{[1,a],[b_k+B/2,b_k+B]} = \dot\bigcup_{b \in [b_k+B/2,b_k+B]} R_{[1,a],[b,b]}.
\end{equation*}
By hypothesis, $|\psi| > 1$ on an at most $\ep a^{3/2} (\log a)^{-1/2}$ many points in $4 R_{[1,a],[1,a]} \setminus F$. On the other hand, there are at least $c B a \geq c a^{3/2} (\alpha' \log a)^{-1/2}$ many points in $R_{[1,a],[b_k+B/2,b_k+B]}$ and at most $\ep c B a \leq C \ep \alpha^{3/2} (\alpha' \log a)^{-1/2}$ many points in $R_{[1,a],[b_k+B/2,b_k+B]} \cap F$.  It follows that, since $\ep \leq c \ep' (\alpha')^{-1/2}$, there is a $b_{k+1} \in [b_k+B/2,b_k+B]$ such that $|\psi| \leq 1$ on a $1 - \ep'$ fraction of $R_{[1,a],[b_{k+1}-1,b_{k+1}]}.$

With the claim in hand, we apply $\mathcal E_{\mathrm{ni}}(R_{[1,a],[b_k-1,b_{k+1}]})$ to conclude
\begin{equation*}
\| \psi \|_{\ell^\infty(R_{[1,a],[b_k-1,b_{k+1}]})} \leq e^{\alpha' (b_{k+1} - b_k) \log a} (1 + \| \psi \|_{\ell^\infty(R_{[1,a],[b_k-1,b_k]})}).
\end{equation*}
By induction, we obtain $\|\psi\|_{\ell^\infty(R_{[1,a],[1,2a]})} \leq (e^{C \alpha' B \log a})^{C a / B} \leq e^{\alpha a \log a}$.
\end{proof}

\subsection{Covering argument}

\tref{continuation} is proved using a Calderon--Zygmund stopping time argument.  This is a random version of \cite{Buhovsky-Logunov-Malinnikova-Sodin}*{Proposition 3.9}.

\begin{proof}[Proof of \tref{continuation}]
For $Q \subseteq \Z^2$ a tilted square and $\beta > 1 > \delta > 0$, let $\mathcal E_{\mathrm{uc}}'(Q,F)$ denote the event that
\begin{equation}
\label{e.ucprimeh}
\begin{cases}
|\lambda - \bar \lambda| \leq e^{-\beta (\ell(Q) \log \ell(Q))^{1/2}} \\
H \psi = \lambda \psi \mbox{ in } Q \\
|\psi| \leq 1 \mbox{ in a $1-\delta (\ell(Q) \log(Q))^{-1/2}$ fraction of $Q \setminus F$},
\end{cases}
\end{equation}
implies $|\psi| \leq e^{\beta \ell(Q) \log \ell(Q)}$ in $\tfrac1{64} Q$.  By a covering argument, it is enough to prove that, if $F$ is $\delta$-regular in $Q$ and $\ell(Q) \geq \beta$, then $\P[\mathcal E_{\mathrm{uc}}' | V_F = v] \geq 1 - e^{-\delta \ell(Q)^{1/4}}$.  Indeed, for any square $Q \subseteq \Z^2$ we can find a list of tilted squares $Q_1, ..., Q_N \subseteq Q$ such that $\frac{1}{2} Q \subseteq \cup_k \frac{1}{64} Q_k$, $N \leq C$, and $\min_k \ell(Q_k) \geq c \ell(Q)$.  Now, if the conclusion of the theorem holds for each $Q_k$ then it is also true for $Q$.

Let $\alpha > 1 > \ep > 0$ denote a valid pair of constants for \lref{excellent}.  We may assume $\ep^2 > \delta$ and $\beta \geq 2 \alpha$.  Let $\mathcal Q$ denote the set of tilted squares $Q' \subseteq Q$ such that 
\begin{enumerate}
\item $\ell(Q') \geq \ell(Q)^{1/4}$
\item $2Q' \subseteq \tfrac12 Q$
\item $F$ is $\ep$-sparse in $2Q'$
\item $\tfrac14 Q' \cap \tfrac{1}{64} Q \neq \emptyset$.
\end{enumerate}
By a union bound, the event
\begin{equation*}
\mathcal E_{\mathrm{ex}} = \bigcap_{Q' \in \mathcal Q} \mathcal E_{\mathrm{ex}}(Q',F) \cap \mathcal E_{\mathrm{ex}}(2Q',F)
\end{equation*}
satisfies $\P[\mathcal E_{\mathrm{ex}} | V_F = v] \geq 1 - e^{-\ep \ell(Q)^{1/4} + C \log \ell(Q)} \geq 1 - e^{-\delta \ell(Q)^{1/4}}$.  Thus, it suffices to prove $\mathcal E'_{\mathrm{uc}}(Q,F) \supseteq \mathcal E_{\mathrm{ex}}$.

Henceforth we assume that $\mathcal E_{\mathrm{ex}}$ and \eref{ucprimeh} hold.  Our goal is to prove $|\psi| \leq e^{\beta \ell(Q) \log \ell(Q)}$ in $\tfrac{1}{64} Q$.  Let $\mathcal Q_{\mathrm{g}} \subseteq \mathcal Q$ denote the set of tilted squares $Q' \in \mathcal Q$ such that
\begin{equation*}
\|\psi\|_{\ell^\infty(Q')} \leq e^{\beta \ell(Q') \log \ell(Q')}
\end{equation*}
Let $\mathcal Q_{\mathrm{mg}} \subseteq \mathcal Q_{\mathrm{g}}$ denote the $Q' \in \mathcal Q_{\mathrm{g}}$ that are maximal with respect to inclusion.

\begin{claim}
\label{cl.alternatives}
If $Q' \in \mathcal Q_{\mathrm{mg}}$, then one of the following holds.
\begin{enumerate}
\item $4Q' \nsubseteq \tfrac12 Q$
\item $F$ is not $\delta$-sparse in $4Q'$
\item $|\{ |\psi| \geq 1 \} \cap 4Q'| \geq \ep (\ell(Q) \log \ell(Q))^{-1/2} |4Q'|.$
\end{enumerate}
\end{claim}

Suppose $Q' \in \mathcal Q_{\mathrm{g}}$ and all three conditions are false.  Observe that $4 Q' \subseteq Q$ and, since $F$ is $\delta$-sparse in $4 Q'$ and we have assumed that $\delta \leq \ep^2 \leq \ep$, we see that $F$ is $\ep$-sparse in $4 Q'$. Thus, the event $\mathcal E_{\mathrm{ex}}(2Q',F)$ holds and we see that
\begin{equation*}
\begin{aligned}
\| \psi \|_{\ell^\infty(2Q')}
& \leq e^{\alpha \ell(2Q') \log \ell(2Q')} \max \{ 1, \| \psi \|_{\ell^\infty(Q')} \} \\
& \leq e^{\alpha \ell(2Q') \log \ell(2Q')} e^{\beta \ell(Q') \log \ell(Q')} \\
& \leq e^{\beta \ell(2Q') \log \ell(2Q')}.
\end{aligned}
\end{equation*}
Here we used $\beta \geq 2 \alpha$.  Since we also have $\tfrac14 (2Q') \cap \tfrac{1}{64} Q \supseteq \tfrac14 Q' \cap \tfrac{1}{64} Q \neq \emptyset$, we see that $2Q' \in \mathcal Q_{\mathrm{g}}$.  In particular, $Q'$ is not maximal with respect to inclusion.

\begin{claim}
\label{cl.alternatives23small}
If $\mathcal Q_{\mathrm{mg}}^k \subseteq \mathcal Q$, for $k = 1, 2, 3$, denote the sets of maximal good squares where the three alternatives from the previous claim hold, then $|\cup \mathcal Q_{\mathrm{mg}}^2| + |\cup \mathcal Q_{\mathrm{mg}}^3| \leq C (\delta + \delta/\ep) |Q|$.  Here $\cup \mathcal Q_{\mathrm{mg}}$ is shorthand for $\cup_{Q \in \mathcal Q_{\mathrm{mg}}} Q$.
\end{claim}

By the Vitali covering lemma (see for example Stein \cite{Stein}*{Section 3.2}), we can find $\tilde {\mathcal Q}_{mg}^k \subseteq \mathcal Q_{\mathrm{mg}}^k$ such that $|\cup \tilde {\mathcal Q}_{mg}^k| \geq c |\cup \mathcal Q_{\mathrm{mg}}^k|$ and $Q', Q'' \in \tilde {\mathcal Q}_{mg}^k$ implies $4 Q' \cap 4 Q'' = \emptyset$.  Since $F$ is $\delta$-regular, we must have $|\cup \tilde {\mathcal Q}_{mg}^2| \leq C \delta |Q|$.  Since $|\psi| \leq 1$ on a $1-\delta (\ell(Q) \log \ell(Q))^{-1/2}$ fraction of $Q$, we must have $|\cup \tilde {\mathcal Q}_{mg}^3| \leq C (\delta/\ep) |Q|$.

\begin{claim}
\label{cl.lotsofgoodsquares}
If $\delta > 0$ is sufficiently small, then $|\cup \mathcal Q_{\mathrm{g}}| \geq c |Q|$.
\end{claim}

Observe that we can find a list of tilted squares $Q'_1, ..., Q'_K \subseteq \tfrac{1}{64} Q$ such that $2 Q'_k$ are disjoint, $\ell(Q'_k) \geq \ell(Q)^{1/4}$, and $K \geq c \ell(Q)^{3/2}$.  Since $F$ is $\delta$-regular, we have $F$ is $\delta$-sparse in $2 Q'_k$ for a $1-C\delta$ fraction of the $Q'_k$.  Since $|\psi| > 1$ on at most $\delta \ell(Q)^{3/2} (\log \ell(Q))^{-1/2}$ many points in $Q$ and there are at least $c \ell(Q)^{3/2}$ squares $Q'_k$, we must have $|\psi| \leq 1$ on at least half of the $Q'_k$.  Assuming $\delta > 0$ is sufficiently small, we see that $\mathcal Q_{\mathrm{g}}$ contains at least $c |Q|^{3/2}$ disjoint squares with volume $|Q|^{1/2}$.

\begin{claim}
$|\psi| \leq e^{\beta \ell(Q) \log \ell(Q)}$ in $\tfrac{1}{64} Q$.
\end{claim}

Since we assumed $\ep^2 > \delta > 0$ and $\ep > 0$ small, \clref{alternatives23small} and \clref{lotsofgoodsquares} together imply there is a $Q' \in \mathcal Q_{\mathrm{mg}}^1 \setminus (\mathcal Q_{\mathrm{mg}}^2 \cup \mathcal Q_{\mathrm{mg}}^3)$.    Together $Q' \in \mathcal Q_{\mathrm{g}}$, $Q' \notin \mathcal Q_{\mathrm{mg}}^2$, $Q' \notin \mathcal Q_{\mathrm{mg}}^3$, and $\mathcal E_{\mathrm{ex}}(2Q',F)$ imply that $\| \psi \|_{\ell^\infty(2Q')} \leq e^{\beta \ell(2Q') \log \ell(2Q')}$.  Since $4Q' \nsubseteq \tfrac12 Q$ and $Q' \cap \tfrac{1}{64} Q \neq \emptyset$, we have $\tfrac{1}{64} Q \subseteq 2 Q'$ and the claim.
\end{proof}

\section{Sperner's Theorem}

We prove a generalization of Sperner's theorem \cite{Sperner}.  Our argument is a modification of Lubell's proof \cite{Lubell} of the Lubell--Yamamoto--Meshalkin inequality.  Recall that a Sperner family is a set $\mathcal A$ of subsets of $\{ 1, ..., n \}$ that form an antichain with respect to inclusion.  We consider a relaxation of this condition.

\begin{definition}
Suppose $\rho \in (0,1]$.  A set $\mathcal A$ of subsets of $\{ 1, ..., n \}$ is $\rho$-Sperner if, for every $A \in \mathcal A$, there is a set $B(A) \subseteq \{ 1, ..., n \} \setminus A$ such that $|B(A)| \geq \rho (n - |A|)$ and $A \subseteq A' \in \mathcal A$ implies $A' \cap B(A) = \emptyset$.
\end{definition}

Note that a Sperner family is $1$-Sperner with $B(A) = \{ 1, ..., n \} \setminus A$.  In particular, the following result is a generalization of Sperner's theorem.

\begin{theorem}
\label{t.sperner}
If $\rho \in (0,1]$ and $\mathcal A$ is a $\rho$-Sperner set of subsets of $\{ 1, ..., n \}$, then
\begin{equation*}
|\mathcal A| \leq 2^n n^{-1/2} \rho^{-1}.
\end{equation*}
\end{theorem}

\begin{proof}
Let $\Pi_n$ denote the set of permutations of $\{ 1, ..., n \}$.  For $\sigma \in \Pi_n$, let $\mathcal A_\sigma = \{ \{ \sigma_1, ..., \sigma_k \} \in \mathcal A : k = 0, ..., n \}$.      For $k \geq 0$, let $\mathcal A_k = \{ A \in \mathcal A : |A| = k \}$.

\begin{claim}
$\left(\begin{smallmatrix} n \\ k \end{smallmatrix} \right) \leq 2^n n^{-1/2}$ for $k = 0, ..., n$.
\end{claim}

This is standard.	

\begin{claim}
If $A \in \mathcal A_k$, then $|\{ \sigma \in \Pi_n : A \in \mathcal A_\sigma \}| = k! (n-k)!$.
\end{claim}

This follows because $A \in \mathcal A_\sigma$ if and only if $\sigma$ is a permutation of $A$ concatenated with a permutation of $\{ 1, ..., n \} \setminus A$.

\begin{claim}
If $j \geq 0$, then $|\{ \sigma \in \Pi_n : |\mathcal A_\sigma| \geq j + 1)\}| \leq n! (1 - \rho)^j$.
\end{claim}

Sample a uniform random $\sigma \in \Pi_n$ one element at a time.  In order to have $|\mathcal A_\sigma| \geq j + 1$, there must be a least $k \geq 0$ such that $\{ \sigma_1, ..., \sigma_k \} \in \mathcal A$.  Moreover, by the $\rho$-Sperner property, it must also be the case that $\sigma_{k+i} \notin B(\{ \sigma_1, ..., \sigma_k\})$ for $i = 1, ..., j$.  Each time we sample the next $\sigma_{k+i}$, the probability that $\sigma_{k+i} \notin B(\{ \sigma_1, ..., \sigma_k\})$ is at most $1-\rho$.  In particular, the probability a uniform random $\sigma \in \Pi_n$ has $|\mathcal A_\sigma| \geq j + 1$ is at most $(1-\rho)^j$.

Using the claims, compute
\begin{equation*}
\begin{aligned}
|\mathcal A|
& = \sum_{k \geq 0} |\mathcal A_k| \\
& \leq 2^n n^{-1/2} \sum_{k \geq 0} \frac{k! (n-k)!}{n!} |\mathcal A_k| \\
& = 2^n n^{-1/2} \sum_{\sigma \in \Pi_n} \frac{1}{n!} |\mathcal A_\sigma| \\
& = 2^n n^{-1/2} \sum_{j \geq 0} \frac{1}{n!} |\{ \sigma \in \Pi_n : |\mathcal A_\sigma| \geq j + 1 \}| \\
& \leq 2^n n^{-1/2} \sum_{j \geq 0} (1-\rho)^j \\
& = 2^n n^{-1/2} \rho^{-1}.
\end{aligned}
\end{equation*}
Here the second, third, and fifth steps follow from claims, in order.	
\end{proof}

\section{Wegner Estimate}

We recall the Courant--Fischer--Weyl min-max principle, which says that, for $A \in S^2(\R^n)$, the eigenvalues $\lambda_1(A) \geq \cdots \geq \lambda_n(A)$ can be computed by
\begin{equation*}
\lambda_k(A) = \max_{\substack{V \subseteq \R^n \\ \dim(V) = k}} \min_{\substack{v \in V \\ \| v \| = 1}} \langle v, A v \rangle = \min_{\substack{V \subseteq \R^n \\ \dim(V) = n-k+1}} \max_{\substack{v \in V \\ \| v \| = 1}} \langle v, A v \rangle.
\end{equation*}
We use this to prove the following eigenvalue variation result.

\begin{lemma}
\label{l.minmax}
Suppose the real symmetric matrix $A \in S^2(\R^n)$ has eigenvalues $\lambda_1 \geq \lambda_2 \geq \cdots \geq \lambda_n \in \R$ with orthonormal eigenbasis $v_1, v_2, ..., v_n \in \R^n$. If
\begin{enumerate}
\item $0 < r_1 < r_2 < r_3 < r_4 < r_5 < 1$
\item $r_1 \leq c \min \{ r_3 r_5, r_2 r_3 / r_4 \}$
\item $0 < \lambda_j \leq \lambda_i < r_1 < r_2 < \lambda_{i-1}$
\item $v_{j,k}^2 \geq r_3$
\item $\sum_{r_2 < \lambda_\ell < r_5} v_{\ell,k}^2 \leq r_4$
\end{enumerate}
then $\trace \id_{[r_1,\infty)}(A) < \trace \id_{[r_1,\infty)}(A + e_k \otimes e_k)$, where $e_k \in \R^n$ is the $k$th standard basis element.
\end{lemma}

\begin{proof}
It is enough to show $\lambda_i(A') \geq r_1$, where $A' = A + e_k \otimes e_k$.  

The min-max principle gives
\begin{equation*}
\lambda_i(A + e_k \otimes e_k)
 = \max_{\substack{W \subseteq \R^n \\ \dim W = i}} \min_{\substack{w \in W \\ \|w\| = 1}} \langle w, A' w \rangle
 \geq \min_{\substack{w \in W_{i,j} \\ \|w\| = 1}} \langle w, A' w \rangle,
\end{equation*}
where $W_{i,j} = \mathrm{Span} \{ v_1, ..., v_{i-1}, v_j \}$.

To estimate the right-hand side above, we observe that every unit vector $w \in W_{i,j}$ can be written $\alpha_1 v_j + \alpha_2 w_2 + \alpha_3 w_3$, where $\alpha_1^2 + \alpha_2^2 + \alpha_3^2 = 1$, $w_2 \in \mathrm{Span}\{ v_\ell : r_2 < \lambda_\ell < r_5 \}$ a unit vector, and $w_3 \in \mathrm{Span} \{ v_\ell : r_5 \leq \lambda_\ell \}$ a unit vector.

We break into three cases.

Case 1.  If $\alpha_2^2 \geq c r_3 / r_4$, then $\langle w, A' w \rangle \geq \alpha_2^2 \langle w_2, A w_2 \rangle \geq \alpha_2^2 r_2 \geq c r_2 r_3 / r_4 \geq r_1$.

Case 2.  If $\alpha_3^2 \geq c r_3$, then $\langle w, A' w \rangle \geq \alpha_3^2 \langle w_3, A w_3 \rangle \geq \alpha_3^2 r_5 \geq c r_3 r_5 \geq r_1$.

Case 3.  If $\alpha_2^2 \leq c r_3 / r_4$ and $\alpha_3^2 \leq c r_3$, then $\langle w, A' w \rangle \geq \langle w, (e_k \otimes e_k) w \rangle = (\alpha_1 v_{j,k} + \alpha_2 w_{2,k} + \alpha_3 w_{3,k})^2 \geq \tfrac12 \alpha_1^2 r_3 - 2 \alpha_2^2 r_4 - 2 \alpha_3^2 \geq c r_3 \geq r_1$.  Here we used the fact that $(x+y+z)^2 \geq \tfrac12 x^2 - 2 y^2 - 2 z^2$ holds for all $x, y, z \in \R$ and that our hypotheses imply $v_{j,k}^2 \geq r_3$ and $w_{2,k}^2 \leq r_4.$
\end{proof}

We need the following bound on the size of a family of almost orthonormal vectors.  An almost identical version of the following lemma and proof appears in Tao \cite{Tao} as a step in a proof of a version of the Kabatjanskii--Levenstein \cite{Kabatjanskii-Levenstein} bound.

\begin{lemma}[Tao \cite{Tao}]
  \label{l.kl}
  If $v_1, ..., v_m \in \R^n$ satisfy $|v_i \cdot v_j - \delta_{ij}| \leq (5n)^{-1/2}$, then $m \leq (5 - \sqrt 5) n / 2$.
\end{lemma}

\begin{proof}
   Let $A \in S^2(\R^m)$ be given by $A_{ij} = v_i \cdot v_j$.  Since the matrix $A$ has rank at most $n$, the matrix $A - I$ has eigenvalue $-1$ with multiplicity at least $m - n$ and Hilbert-Schmidt norm $\|A - I\|_2^2 \geq m - n$.  On the other hand, the hypothesis gives $\|A - I\|_2^2 = \sum_{ij} (v_i \cdot v_j - \delta_{ij})^2 \leq m^2 / (5 n)$.  In particular, $m - n \leq m^2/(5n)$ which implies that $m$ can not be in the interval $((5-\sqrt 5)n/2,(5+\sqrt 5)n/2)$.  Since this interval contains a positive integer whenever $n \geq 1$, we must have that $m$ is less than or equal to its left endpoint.
\end{proof}

We need the following weak unique continuation bound.  The advantage over \tref{continuation} is that it holds a priori.

\begin{lemma}
\label{l.crudecontinuation}
For every integer $K \geq 1$, if
\begin{enumerate}
\item $L \geq C_K L' \geq L' \geq C_K$
\item $Q \subseteq \Z^2$ with $\ell(Q) = L$
\item $Q'_k \subseteq Q$ with $\ell(Q'_k) = L'$ for $k = 1, ..., K$
\item $H_Q \psi = \lambda \psi$,
\end{enumerate}
then
\begin{equation*}
\| \psi \|_{\ell^\infty(Q')} \geq e^{-C_K L'} \| \psi \|_{\ell^\infty(Q)}
\end{equation*}
holds for some $2Q' \subseteq Q \setminus \cup_k Q'_k$ with $\ell(Q') = L'$.
\end{lemma}

\begin{proof}
For simplicity, we assume that $L$ is an integer multiple of $L'$.  Moreover, we assume that $Q$ and $Q'_k$ are aligned squares in the sense that the coordinates of their southwest corner are divisible by their side length.  Henceforth $Q'$ always denotes an aligned square of side length $L'$.

We recall two elementary observations, both of which can be found in Bourgain--Klein \cite{Bourgain-Klein}.  The first is that, if $H \psi = \lambda \psi$ in a rectangle $[0,a] \times [0,b]$, then
\begin{equation*}
\max_{[0,a] \times [0,b]} |\psi| \leq e^{C(a+b)} \max_{[0,a] \times [0,b] \setminus [1,a-2] \times [1,b-2]} |\psi|.
\end{equation*}
The second is that, if $H \psi = \lambda \psi$ in the triangle $T_a = \{ (x,y) : 0 \leq x \leq a - |y| \}$ with base $\partial^b T_a = \{ (x,y) \in T_a : x \leq 1 \}$, then
\begin{equation*}
\max_{T_a} |\psi| \leq e^{Ca} \max_{\partial^b T_a} |\psi|.
\end{equation*}
Both of these observations follow by iterating the equation $H \psi = \lambda \psi$.  We use them to prove the following two claims, from which the lemma easily follows.

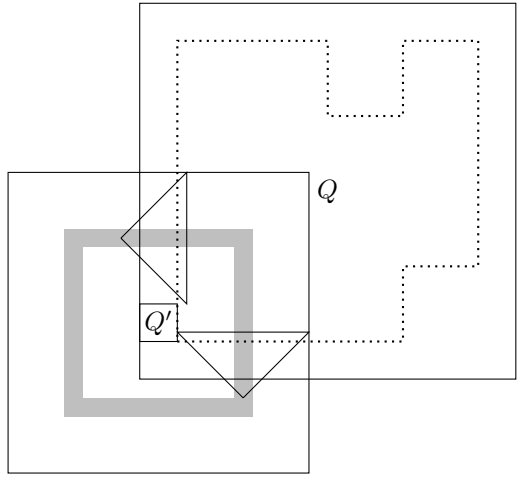
\begin{figure}[t]
\begin{tikzpicture}[scale=1/4]
\fill[lightgray] (-4,-2) rectangle (6,8);
\fill[white] (-3,-1) rectangle (5,7);
\draw (0,0) rectangle (20,20);
\draw (10,10) node {$Q$};
\draw[thick,dotted] (2,2) -- (14,2) -- (14,6) -- (18,6) -- (18,18) -- (14,18) -- (14,14) -- (10,14) -- (10,18) -- (2,18) -- (2,2);
\draw (0,2) rectangle (2,4);
\draw (1,3) node {$Q'$};
\draw (-7,-5) rectangle (9,11);
\draw (-1,7.5) -- (2.5,4) -- (2.5,11) -- (-1,7.5);
\draw (5.5,-1) -- (9,2.5) -- (2,2.5) -- (5.5,-1);
\end{tikzpicture}
\caption{A schematic for the proof of \lref{crudecontinuation}.  The set $E$ is enclosed by the dotted line.  The annulus $5Q' \setminus 4Q'$ is depicted in gray and its intersection with the square $Q$ is covered by the set $E$ and two triangles.}
\label{f.crude}
\end{figure}

\begin{claim}
$\| \psi \|_{\ell^\infty(Q \setminus \cup_k 8 Q'_k)} \leq e^{C L'} \| \psi \|_{\ell^\infty(E)},$
where
\begin{equation*}
E = \bigcup_{2Q' \subseteq Q \setminus \cup_k Q'_k} Q'.
\end{equation*}
\end{claim}

Fix $x \in Q \setminus \cup_k 8 Q'_k$ and assume $x \notin E$. Observe there is a $Q' \subseteq Q$ such that $x \in Q'$, $2Q' \not\subseteq Q$, and $8Q' \cap \cup_k Q'_k = \emptyset$.  Observe that $Q'$ must be adjacent to the boundary of $Q$.  Moreover, observe that $(5Q' \setminus 4Q') \cap Q$ can be covered by $E$ together with two rotation-translations of triangles $T_{\frac74 L'}$ such that the corresponding rotation-translations of the base $\partial^b T_{\frac74 L'}$ is contained in $E$.  In particular, by the second observation, $|\psi| \leq e^{CL'} \| \psi \|_{\ell^\infty(E)}$ in $(5Q' \setminus 4Q') \cap Q$.  Applying the first observation gives $|\psi| \leq e^{CL'} \| \psi \|_{\ell^\infty(E)}$ in $5Q' \cap Q.$  See \fref{crude} for a schematic of this computation.

\begin{claim}
$\| \psi \|_{\ell^\infty(Q)} \leq e^{C_K L'} \| \psi \|_{\ell^\infty(Q \setminus \cup_k 8 Q'_k)}.$
\end{claim}

Fix $x \in 8 Q'_k \cap Q$.  Observe there is a $Q' \subseteq Q$ and an integer $1 \leq n \leq C K L'$ such that $x \in Q'$ and $(n+1) Q' \setminus n Q'$ does not intersect $\cup_k 8 Q_k'$.  Since we assumed $L \geq C K L'$ large, we can apply the first observation to conclude $|\psi| \leq e^{C K L'} \| \psi \|_{\ell^\infty(Q \setminus \cup_k 8 Q'_k)}$ on $(n+1) Q'$.
\end{proof}

We now prove our analogue of the Wegner estimate  \cite{Bourgain-Kenig}*{Lemma 5.1}.  This brings together all of our new ingredients.

\begin{lemma}
\label{l.wegner}
If
\begin{enumerate}
\item $\eta > \ep > \delta > 0$ small
\item $K \geq 1$ integer
\item $L_0 \geq \cdots \geq L_5 \geq C_{\eta,\ep,\delta,K}$ dyadic scales with $L_k^{1-2\delta} \geq L_{k+1} \geq L_k^{1-\ep}$
\item $Q \subseteq \Z^2$ with $\ell(Q) = L_0$
\item $Q'_1, ..., Q'_K \subseteq Q$ with $\ell(Q'_k) = L_3$
\item $G \subseteq \cup_k Q'_k$ with $0 < |G| \leq L_0^{\delta}$.
\item $F \subseteq \Z^2$ is $\eta$-regular in every $Q' \subseteq Q \setminus \cup_k Q'_k$ with $\ell(Q') = L_3$
\item $V : Q \to [0,1]$, $V_F = v$, $|\lambda - \bar \lambda| \leq e^{-L_5}$, and $H_Q \psi = \lambda \psi$ implies
\begin{equation*}
e^{L_4} \| \psi \|_{\ell^2(Q \setminus \cup_k Q'_k)} \leq \| \psi \|_{\ell^2(Q)} \leq (1 + L_0^{-\delta}) \| \psi \|_{\ell^2(G)},
\end{equation*}
\end{enumerate}
then
\begin{equation*}
\P[\| R_Q \| \leq e^{L_1} | V_F = v] \geq 1 - L_0^{C \ep - 1/2},
\end{equation*}
where $R_Q = (H_Q - \bar \lambda)^{-1}$ is the resolvent of $H_Q$ at energy $\bar \lambda$.
\end{lemma}

\begin{remark}
The scales in the above lemma have the following interpretations:
\begin{center}
\begin{tabular}{rl}
$L_0$ & large scale \\
$e^{L_1}$ & large scale resolvent bound \\
$e^{-L_2}$ & unique continuation lower bound \\
$L_3$ & small scale \\
$e^{-L_4}$ & localization smallness \\
$e^{-L_5}$ & localization interval \\
$L_0^{\delta}$ & localization support \\
\end{tabular}
\end{center}
These are set up to be compatible with the multiscale analysis in \sref{multiscale}.
\end{remark}

\begin{proof}
Throughout the proof, we allow the constants $C > 1 > c > 0$ to depend on $\eta, \ep, \delta, K$.  Let $\lambda_1(H_Q) \geq \cdots \geq \lambda_{L_0^2}(H_Q)$ denote the eigenvalues of $H_Q$.  Choose eigenvectors $\psi_k(H_Q) \in \R^Q$ such that $\| \psi_k \|_{\ell^2(Q)} = 1$ and $H_Q \psi_k = \lambda_k \psi_k$.  We may assume that $\lambda_k$ and $\psi_k$ are deterministic functions of the potential $V_Q \in [0,1]^Q$.

\begin{claim}
We may assume $\cup_k Q'_k \subseteq F$.
\end{claim}

Let $F' = \cup_k Q'_k \setminus F$ and observe that
\begin{equation*}
\P[\mathcal E | V_F = v] = 2^{-|F'|} \sum_{v' : F' \to \{ 0, 1 \}} \P[\mathcal E | V_{F \cup F'} = v \cup v']
\end{equation*}
holds for all events $\mathcal E$.  Thus, it suffices to estimate each term in the sum.

\begin{claim}
$\P[\mathcal E_{\mathrm{uc}}|V_F = v] \geq 1 - e^{-L_0^\ep},$ where $\mathcal E_{\mathrm{uc}}$ denotes the event that
\begin{equation*}
|\{ |\psi| \geq e^{-L_2} \| \psi \|_{\ell^\infty(Q)} \} \setminus F| \geq L_4^{3/2}
\end{equation*}
holds whenever $|\lambda - \bar \lambda| \leq e^{-L_5}$ and $H_Q \psi = \lambda \psi$.
\end{claim}

Let $\alpha' > 1 > \ep' > 0$ be constants that work in \tref{continuation}.  We may assume $\ep' > \eta$.  By \tref{continuation}, the event
\begin{equation*}
\mathcal E_{\mathrm{uc}}' = \bigcap_{\substack{Q' \subseteq Q \setminus \cup_k Q'_k \\ \ell(Q') = L_3}} \mathcal E_{\mathrm{uc}}(Q',F)
\end{equation*}
satisfies
\begin{equation*}
\P[\mathcal E_{\mathrm{uc}}' | V_F = v] \geq 1 - e^{-\ep' L_3^{1/4} - C \log L_0} \geq 1 - e^{-L_0^\ep}.
\end{equation*}
Thus, it suffices to show $\mathcal E_{\mathrm{uc}} \supseteq \mathcal E_{\mathrm{uc}}'$.

Let us suppose that $\mathcal E_{\mathrm{uc}}'$ holds, $|\lambda - \bar \lambda| \leq e^{-L_5}$, and $H_Q \psi = \lambda \psi$.

\lref{crudecontinuation} provides an $L_3$-square $Q' \subseteq Q \setminus \cup_k Q'_k$
\begin{equation*}
\| \psi \|_{\ell^\infty(\frac12 Q')} \geq e^{-C L_3} \| \psi \|_{\ell^\infty(Q)}.
\end{equation*}
Since $\mathcal E_{\mathrm{uc}}(Q',F)$ holds and $|\lambda - \bar \lambda| \leq e^{-L_5} \leq e^{-\alpha' (L_3 \log L_3)^{1/2}}$, we see that
\begin{equation*}
|\{ |\psi| \geq e^{-\alpha' L_3 \log L_3} \| \psi \|_{\ell^\infty(\frac12 Q')} \} \cap Q' \setminus F| \geq \ep' L_3^{3/2} (\log L_3)^{-1/2}.
\end{equation*}
Thus
\begin{equation*}
|\{ |\psi| \geq e^{- L_2} \| \psi \|_{\ell^\infty(Q)} \} \cap Q \setminus F| \geq L_4^{3/2},
\end{equation*}
which proves the inclusion and the claim.

\begin{claim}
For $1 \leq k_1 \leq k_2 \leq L_0^2$ and $0 \leq \ell \leq C L_0^{\delta}$, we have
\begin{equation*}
\P[\mathcal E_{k_1,k_2,\ell} \cap \mathcal E_{\mathrm{uc}} | V_F = v] \leq C L_0 L_4^{-3/2}
\end{equation*}
where $\mathcal E_{k_1,k_2,\ell}$ denotes the event that
\begin{equation*}
|\lambda_{k_1} - \bar \lambda|, |\lambda_{k_2} - \bar \lambda| < s_\ell
\quad \mbox{and} \quad |\lambda_{k_1-1} - \bar \lambda|, |\lambda_{k_2+1} - \bar \lambda| \geq s_{\ell+1},
\end{equation*}
and $s_\ell = e^{-L_1 + (L_2 - L_4 + C) \ell}$.
\end{claim}

Since we are conditioning on $V_F = v$, we can view the events $\mathcal E_{\mathrm{uc}}$ and $\mathcal E_{k_1, k_2, \ell}$ as subsets of $\{ 0, 1 \}^{Q \setminus F}$.  For $1 \leq k_1 \leq k_2 \leq L^2$ and $i = 0, 1$, let $\mathcal E_{k_1,k_2,\ell,i}$ denote the event that
\begin{equation*}
\mathcal E_{k_1,k_2,\ell} \quad \mbox{and} \quad |\{ |\psi_{k_1}| \geq e^{-L_2} \} \cap \{ V_Q = i \} \setminus F| \geq \tfrac12 L_4^{3/2}.
\end{equation*}
Observe that
\begin{equation*}
\mathcal E_{k_1,k_2,\ell} \cap \mathcal E_{\mathrm{uc}} \subseteq \mathcal E_{k_1,k_2,\ell,0} \cup \mathcal E_{k_1,k_2,\ell,1}.
\end{equation*}
Next, observe that if $w \in \mathcal E_{k_1,k_2,\ell,i}$, $x \in Q \setminus F$, $w(x) = i$, and $|\psi_{k_1}(x)| \geq e^{-L_2}$, then $w^x \notin \mathcal E_{k_1,k_2,\ell,i}$, where
\begin{equation*}
w^x(y) = \begin{cases} w(y) & \mbox{if } y \neq x \\ 1-w(y) & \mbox{if } y = x. \end{cases}
\end{equation*}
In the case $i = 0$, this follows by applying \lref{minmax} centered at $H_Q - \bar \lambda + s_\ell$ at site $x$ and with radii $r_1 = 2 s_\ell$, $r_2 = s_{\ell+1}$, $r_3 = e^{-L_2}$, $r_4 = e^{-c L_4}$, and $r_5 = e^{-L_5}$.  To guarantee that the hypotheses of \lref{minmax} hold, we use hypotheses (8) to observe that $\sum_{|\lambda_i - \bar \lambda| < e^{-L_5}} \psi_i(x)^2 < e^{-c L_4}$.  Lemma 5.1 implies that $\lambda_{k_1}$ moves out of the interval $(\bar \lambda - s_\ell, \bar \lambda + s_\ell)$ when $w(x)$ is changed from $0$ to $1$.  The case $i = 1$ is symmetric.

By definition of $\mathcal E_{k_1,k_2,\ell,i}$, if $w \in \mathcal E_{k_1,k_2,\ell,i}$, then the set $B_i(w)$ of $x \in Q \setminus F$ for which $w(x) = i$ and $|\psi_{k_1}(x)| \geq e^{-L_2}$ has size at least $\tfrac12 L_4^{3/2}$.  Since $Q \setminus F$ has size at most $L_0^2$, we see that $\mathcal E_{k_1,k_2,\ell,i}$ is $\tfrac12 L_0^{-2} L_4^{3/2}$-Sperner via the $B_i$.  Applying \tref{sperner}, we obtain $\P[\mathcal E_{k_1,k_2,\ell,i} | V_F = v] \leq C L_0 L_4^{-3/2}.$  This gives the claim.

\begin{claim}
There exists a set $K \subseteq \{ 1, ..., L_0^2 \}$, depending only on $F$ and $v$, such that $|K| \leq C L_0^{\delta}$ and
\begin{equation*}
\{ \| R_Q \| > e^{L_1} \} \cap \{ V_F = v \} \subseteq \bigcup_{\substack{k_1, k_2 \in K \\ 0 \leq \ell \leq C L_0^{\delta}}} \mathcal E_{k_1,k_2,\ell}.
\end{equation*}
\end{claim}

Since we are conditioning on $V_F = v$, we can view $\lambda_k$ and $\psi_k$ as functions on $\{ 0, 1 \}^{Q \setminus F}$.  Let $1 \leq k_1 < \cdots < k_m \leq L_0^2$ list all indices $k_i$ for which there is at least one $w \in \{ 0, 1 \}^{Q \setminus F}$ such that $|\lambda_k(w) - \bar \lambda| \leq e^{-L_2}$.  To prove the claim, it suffices to prove that $m \leq C L_0^{\delta}$.  Indeed, we can always find an $0 \leq \ell \leq m$ such that the annulus $[\bar \lambda -s_{\ell+1}, \bar \lambda + s_{\ell+1}] \setminus [\bar \lambda - s_\ell, \bar \lambda + s_\ell]$ contains no eigenvalue of $H_Q$.

Since $\cup_k Q_k' \subseteq F$, the left-hand side of hypothesis (8) says that $w \in \{ 0, 1 \}^{Q \setminus F}$ and $|\lambda_{k_i}(w) - \bar \lambda| \leq e^{-L_5}$ imply $\| \psi_{k_i}(w) \|_{\ell^\infty(Q \setminus F)} \leq e^{-L_4}$.  In particular, if there is a $w_0 \in \{ 0, 1 \}^{Q \setminus F}$ such that $|\lambda_{k_i}(w_0) - \bar \lambda| \leq e^{-L_2}$, then, by eigenvalue variation,
\begin{equation*}
|\lambda_{k_i}(w) - \bar \lambda| \leq e^{-L_4}
\end{equation*}
holds for all $w \in \{ 0, 1 \}^{Q \setminus F}$.  Indeed, for $w_t = w_0 + t (w - w_0)$ and $t \in [0,1]$, we compute
\begin{equation*}
\begin{aligned}
|\lambda_{k_i}(w_t) - \bar \lambda| & \leq |\lambda_{k_i}(w_0) - \bar \lambda| + \int_0^t \|\psi_{k_i}(w_s) \|_{\ell^2(Q \setminus F)}^2 \,ds \\
& \leq e^{-L_2} + \int_0^t |Q| e^{-2 L_4} + \id_{|\lambda_{k_i}(w_s)-\bar \lambda| \geq e^{-L_5}}\,ds \\
& \leq e^{-L_4} + \id_{\max_{0 \leq s \leq t} |\lambda_{k_i}(w_s) - \bar \lambda| \geq e^{-L_5}}
\end{aligned}
\end{equation*}
and conclude by continuity.

The right-hand side of hypothesis (8) now also tells us that, for all $w \in \{ 0, 1 \}^{Q \setminus F}$, we have $1 = \| \psi_{k_i}(w) \|_{\ell^2(Q)} \geq \| \psi_{k_i}(w) \|_{\ell^2(G)} \geq 1 - C L_0^{-\delta}$.  In particular, we have $|\langle \psi_{k_i}(w), \psi_{k_j}(w) \rangle_{\ell^2(G)} - \delta_{ij}| \leq C  L_0^{-\delta} \leq (5|G|)^{-1/2}$.  \lref{kl} now tells us that $m \leq C |G| \leq C L_0^{\delta}$.

We compute
\begin{equation*}
\begin{aligned}
& \P[\|R_Q\| > e^{L_1} | V_F = v] \\
& \leq \P[\mathcal E_{\mathrm{uc}}^c | V_F = v] + \sum_{\substack{k_1,k_2 \in K \\ 1 \leq \ell \leq C L_0^{\delta}}} \P[\mathcal E_{k_1, k_2,\ell} \mbox{ and } \mathcal E_{\mathrm{uc}} | V_F = v] \\
& \leq e^{-L_0^\ep} + C L_0^{1+3 \delta} L_4^{-3/2} \\
& \leq L_0^{10 \ep - 1/2}
\end{aligned}
\end{equation*}
using the above claims.
\end{proof}

\section{The Geometric Resolvent Identity}

We prove a discrete analogue of \cite{Bourgain-Kenig}*{Section 2}, which encapsulates the deterministic part of the multiscale analysis.  We need the following consequence of the geometric resolvent identity.

\begin{lemma}
If $x \in Q' \subseteq Q$ and $y \in Q$, then there are $u \in Q'$ and $v \in Q \setminus Q'$ such that $|u - v| = 1$ and $|R_Q(x,y)| \leq |R_{Q'}(x,y)| + |Q| |R_{Q'}(x,u)| |R_Q(v,y)|$.
\end{lemma}

\begin{proof}
Write
\begin{equation*}
(H - \bar \lambda)_Q = (H - \bar \lambda)_{Q'} + (H - \bar \lambda)_{Q \setminus Q'} - \id_{Q'} \Delta \id_{Q \setminus Q'} - \id_{Q \setminus Q'} \Delta \id_{Q'}.
\end{equation*}
Multiply on the left by $R_{Q'}$ and on the right by $R_{Q}$ to obtain
\begin{equation*}
R_{Q'} = \id_{Q'} R_Q - R_{Q'} \Delta \id_{Q \setminus Q'} R_{Q}.
\end{equation*}
Expanding the last term in coordinates gives
\begin{equation*}
R_Q(x,y) = R_{Q'}(x,y) + \sum_{\substack{u \in Q' \\ v \in Q \setminus Q' \\ |u - v| = 1}} R_{Q'}(x,u) R_Q(v,y).
\end{equation*}
Since the sum has at most $|Q|$ terms, the lemma follows.
\end{proof}

We show propagation of exponential decay from small scales to large scales, even in the presence of finitely many defects.

\begin{lemma}
\label{l.multiscale}
If
\begin{enumerate}
\item $\ep > \delta > 0$ are small
\item $K \geq 1$ an integer
\item $L_0 \geq \cdots \geq L_6 \geq C_{\ep,\delta,K}$ dyadic scales with $L_k^{1-\ep} \geq L_{k+1}$
\item $1 \geq m \geq 2 L_5^{-\delta}$ represents the exponential decay rate
\item $Q \subseteq \Z^2$ a square with $\ell(Q) = L_0$
\item $Q'_1, ..., Q'_K \subseteq Q$ disjoint $L_2$-squares with $|R_{Q'_k}| \leq e^{L_4}$
\item for all $x \in Q$ one of the following holds
\begin{enumerate}
\item there is a $Q'_k$ such that $x \in Q'_k$ and $\dist(x,Q \setminus Q'_k) \geq \tfrac18 \ell(Q'_k)$
\item there is an $L_5$-square $Q'' \subseteq Q$ such that $x \in Q''$, $\dist(x,Q \setminus Q'') \geq \tfrac18 \ell(Q'')$, and $|R_{Q''}(y,z)| \leq e^{L_6 - m |y-z|}$ for $y, z \in Q''$.
\end{enumerate}
\end{enumerate}
then $|R_Q(x,y)| \leq e^{L_1 - \tilde m |x-y|}$ for $x, y \in Q$ where $\tilde m = m - L_5^{-\delta}$.
\end{lemma}

\begin{remark}
The scales in the above lemma have the following interpretations:
\begin{center}
\begin{tabular}{rl}
$L_0$ & large scale \\
$e^{L_1}$ & large scale resolvent bound \\
$L_2$ & defect scale \\
$-L_3$ & defect edge weight \\
$e^{L_4}$ & defect resolvent bound \\
$L_5$ & small scale \\
$e^{L_6}$ & small scale resolvent bound \\
$2L_5^{-\delta}$ & exponential decay bound \\
$L_5^{-\delta}$ & exponential decay loss \\
\end{tabular}
\end{center}
These are set up to be compatible with the multiscale analysis below.
\end{remark}

\begin{proof}
Throughout the proof, we let $C > 1 > c > 0$ depend on $\ep, \delta, K$.

Put a weighted directed multigraph structure on $Q$ as follows.  If $x, y \in Q$, then add the edge $x \to y$ with weight $|x-y|$.  If $x \in Q'_k$, $y \in Q \setminus Q'_k$, $\dist(x, Q \setminus Q'_k) \geq \ell(Q'_k)/8$, and $\dist(y,Q'_k) = 1$, then add the edge $x \to y$ with negative weight $- L_3$.  Since $\ell(Q'_k) = L_2 \gg L_3$ and the $Q_k'$ disjoint, this directed graph has no cycles with negative total weight.  In particular, there is a well defined directed weighted distance $d(x,y)$ that satisfies
\begin{equation*}
d(x,y) \leq d(x,z) + d(z,y)
\end{equation*}
and, using the finiteness of $K$, 
\begin{equation*}
|x-y| \geq d(x,y) \geq |x-y| - C L_2.
\end{equation*}

We estimate the quantity
\begin{equation*}
\alpha = \max_{x,y \in Q} e^{\tilde m d(x,y)} |R_Q(x,y)|
\end{equation*}
in terms of itself.

Suppose $x, y \in Q$ and that $x$ falls into case (a) in the hypotheses.  Using the geometric resolvent identity, we can find points $u \in Q'_k$ and $v \in  Q \setminus Q'_k$ such that $|u - v| = 1$ and 
\begin{equation*}
|R_Q(x,y)| \leq |R_{Q'_k}(x,y)| + L^2 |R_{Q'_k}(x,u)| |R_Q(v,y)|.
\end{equation*}
Using the definition of $\alpha$, $d(x,y) \leq -L_3 + d(v,y)$, and $\tilde m L_3 \geq c L_3^{1-\delta} \gg L_4$, estimate
\begin{equation*}
\begin{aligned}
|R_Q(x,y)|
& \leq e^{L_4} \id_{Q'_k}(y) + L^2 \cdot e^{L_4} \cdot \alpha e^{-\tilde m d(v,y)} \\
& \leq e^{L_4} \id_{Q'_k}(y) + L^2 \cdot e^{L_4} \cdot \alpha e^{-\tilde m (d(x,y) + L_3)} \\
& \leq e^{L_4} \id_{Q'_k}(y) + L^2 \cdot e^{L_4 - \tilde m L_3} \cdot \alpha e^{-\tilde m d(x,y)} \\
& \leq e^{L_4} \id_{Q'_k}(y) + \tfrac12 \alpha e^{-\tilde m d(x,y)}.
\end{aligned}
\end{equation*}
Since $\tilde m d(x,y) \leq C L_2$ when $y \in Q'_k$, we obtain
\begin{equation*}
e^{\tilde m d(x,y)} |R_Q(x,y)| \leq e^{C L_2} + \tfrac12 \alpha.
\end{equation*}

Suppose $x, y \in Q$ and that $x$ falls into case (b) in the hypotheses.  Using the geometric resolvent identity, we can find points $u \in Q''$ and $v \in Q \setminus Q''$ such that $|u - v| = 1$ and
\begin{equation*}
|R_Q(x,y)| \leq |R_{Q''}(x,y)| + L_0^2 |R_{Q''}(x,u)| |R_Q(v,y)| .
\end{equation*}
Using the definition of $\alpha$ and $(m-\tilde m) |x - u| \geq c L_5^{1-\delta} \gg L_6$, we estimate
\begin{equation*}
\begin{aligned}
|R_Q(x,y)|
& \leq e^{L_6} \id_{Q''}(y) + L^2 \cdot e^{L_6 - m |x-u|} \cdot \alpha e^{-\tilde m d(v,y)} \\
& \leq e^{L_6} \id_{Q''}(y) + L^2 \cdot e^{L_6 - (m - \tilde m)|x-u| + 1} \cdot \alpha e^{-\tilde m d(x,y)} \\
& \leq e^{L_6} \id_{Q''}(y) + \tfrac12 \alpha e^{-\tilde m d(x,y)}.
\end{aligned}
\end{equation*}
Since $\tilde m d(x,y) \leq C L_5$ if $y \in Q''$, we obtain
\begin{equation*}
e^{\tilde m d(x,y)} |R_Q(x,y)| \leq e^{C L_5} + \tfrac12 \alpha.
\end{equation*}

Combining the above estimates, we see that
\begin{equation*}
\alpha \leq e^{C L_2} + \tfrac12 \alpha.
\end{equation*}
In particular,
\begin{equation*}
|R_Q(x,y)| \leq e^{C L_2 - \tilde m d(x,y)} \leq e^{L_1 - \tilde m |x - y|}
\end{equation*}
for all $x, y \in Q$.
\end{proof}

We also need the continuity of exponential resolvent bounds.

\begin{lemma}
\label{l.continuity}
If $Q \subseteq \Z^2$ a square, $\lambda \in \R$, $\alpha > \beta > 0$, and
\begin{equation*}
|(H_Q - \lambda)^{-1}(x,y)| \leq e^{\alpha - \beta |x-y|} \quad \mbox{for } x, y \in Q,
\end{equation*}
then, for all $|\lambda' - \lambda| < c \beta |Q|^{-1} e^{-\alpha}$,
\begin{equation*}
|(H_Q - \lambda')^{-1}(x,y)| \leq 2 e^{\alpha - \beta |x-y|} \quad \mbox{for } x, y \in Q.
\end{equation*}
\end{lemma}

\begin{proof}
Recall the resolvent identity
\begin{equation*}
(H_Q - \lambda')^{-1} = (H_Q - \lambda)^{-1} + (H_Q - \lambda)^{-1} (\lambda - \lambda') (H_Q - \lambda')^{-1}
\end{equation*}
formally obtained by multiplying the equation
\begin{equation*}
(H_Q - \lambda') = (H_Q - \lambda) + (\lambda - \lambda')
\end{equation*}
on the left and right by $(H_Q - \lambda)^{-1}$ and $(H_Q - \lambda')^{-1}$.  Using exponential decay, we can estimate the operator norm of $(H_Q - \lambda)^{-1}$ by
\begin{equation*}
\| (H_Q - \lambda)^{-1} \| \leq \| (H_Q - \lambda)^{-1} \|_2 \leq C \beta^{-1} |Q|^{1/2} e^\alpha.
\end{equation*}
Thus, if $|\lambda' - \lambda| \leq c \beta |Q|^{-1/2} e^{-\alpha}$, then $\| (H_Q - \lambda)^{-1} (\lambda' - \lambda) \| \leq 1/2$ and we can solve the resolvent identity for $(H_Q - \lambda')^{-1}$ by fixed point iteration.  To obtain the decay estimate, we define
\begin{equation*}
\gamma = \max_{x,y \in Q} e^{\beta |x-y| - \alpha} |(H_Q - \lambda')^{-1}(x,y)|.
\end{equation*}
We can use the resolvent identity, the exponential decay hypothesis, and $|\lambda' - \lambda| \leq \tfrac12 |Q|^{-1} e^{-\alpha}$ to estimate
\begin{equation*}
\begin{aligned}
& |(H_Q - \lambda')^{-1}(x,y)| \\
& \leq e^{\alpha - \beta |x-y|} + |\lambda' - \lambda| \sum_{z \in Q} e^{\alpha - \beta |x-z|} e^{\alpha - \beta |z-y|} \gamma \\
& \leq e^{\alpha - \beta |x-y|} + |\lambda' - \lambda| |Q| e^\alpha e^{\alpha - \beta |x-y|} \gamma \\
& \leq e^{\alpha - \beta |x-y|} + \tfrac12 e^{\alpha - \beta |x-y|} \gamma.
\end{aligned}
\end{equation*}
Dividing through by $e^{\alpha - \beta |x-y|}$ and computing the maximum over $x, y \in Q$, we obtain $\gamma \leq 1 + \tfrac12 \gamma$ and the lemma.
\end{proof}

\section{Principal eigenvalue}

We give a maximum principle version of the results of \cite{Bourgain-Kenig}*{Section 4}.  The base case of the multiscale analysis relies on an estimate of the principal eigenvalue of $H$ on a large square.  The basic idea is that, if the set $\{ V = 1 \}$ is an $R$-net in the square $Q$, then the principal eigenvalue is bounded below by $c R^{-2} (\log R)^{-1}$.  The same argument also yields exponential decay of the Green's function.

\begin{definition}
If $X \subseteq Y \subseteq \Z^2$ and $R > 0$, then $X$ is an $R$-net in $Y$ if $\max_{y \in Y} \min_{x \in X} |y - x| \leq R$.
\end{definition}

In the next lemma, we use the $R$-net property to construct a barrier to bound the Green's function.

\begin{lemma}
\label{l.principal}
If $Q \subseteq \Z^2$ is a square, $R \geq 2$ a distance, $Q \cap \{ V = 1 \}$ is an $R$-net in $Q$, and $L = R^2 \log R$, then
\begin{equation*}
|H_Q^{-1}(x,y)| \leq e^{C L - c L^{-1} |x-y|} \quad \mbox{for } x, y \in Q.
\end{equation*}
\end{lemma}

\begin{proof}
Recall (or see \cite{Evans}*{Exercise 6.14}) that the principal eigenvalue can be computed via
\begin{equation*}
\lambda_\mathrm{min} = \sup_{\psi : Q \to (0,\infty)} \inf_Q \frac{H \psi}{\psi}.
\end{equation*}
To estimate $\lambda_\mathrm{min}$, we construct a test function $\psi$.  Since we only care about the square $Q$, we may assume that $X = \{ V = 1 \}$ is an $R$-net in all of $\Z^2$.  We may also assume that $R \geq C$ is large.

\begin{claim}
There is a $\psi : \Z^2 \to \R$ such that $H \psi \geq c R^{-2}$ and $1 \leq \psi \leq C \log R$.
\end{claim}

Let $G : \Z^2 \to \R$ denote lattice Green's function.  That is, $G$ is the unique function that satisfies $G(0) = 0$, $G \leq 0$, and $-\Delta G = \id_{\{ 0 \}}$.  Observe that there is a universal small $\ep > 0$ such that, if $R \geq C$, the function
\begin{equation*}
\varphi(x) = 1 - G(x) - \ep R^{-2} |x|^2
\end{equation*}
satisfies
\begin{equation*}
(-\Delta + \id_{\{ 0 \}}) \varphi \geq R^{-2} \ep \quad \mbox{and} \quad 1 \leq \varphi \leq C \log R \quad \mbox{in } B_{3R}
\end{equation*}
and
\begin{equation*}
\min_{B_{3R} \setminus B_{2R}} \varphi \geq \max_{B_R} \varphi.
\end{equation*}
We define $\psi : \Z^2 \to \R$ by
\begin{equation*}
\psi(y) = \min_{x \in B_{3R}(y) \cap X} \varphi(y-x).
\end{equation*}
Observe that, since $X$ is an $R$-net and $\min_{B_{3R} \setminus B_{2R}} \varphi \geq \max_{B_R} \varphi$, we have
\begin{equation*}
\psi(y) = \min_{x \in B_{2R}(y) \cap X} \varphi(y-x).
\end{equation*}
Thus, for any $y \in \Z^2$, we can pick an $x \in B_{2R}(y) \cap X$ such that
\begin{equation*}
\psi(y) = \varphi(y - x) \quad \mbox{and} \quad \psi(z) \leq \varphi(z - x) \quad \mbox{for } |z-y| = 1.
\end{equation*}
This implies $1 \leq \psi(y) \leq C \log R$ and $H \psi(y) \geq H \varphi(y - x) \geq \ep R^{-2}$.

\begin{claim}
$0 \leq H_Q^{-1}(x,y) \leq e^{C L - c L^{-1} |x-y|}$ for all $x, y \in \Z^2$.
\end{claim}

Note that, for any $u, v : \Z^2 \to \R$,
\begin{equation*}
H (u v)(x) = u(x) H v(x) - \sum_{|y-x| = 1} (u(y) - u(x)) v(y).
\end{equation*}
It follows that there is a universal small $\ep > 0$ such that, for all $y \in \Z^2$, the function
\begin{equation*}
\rho_y(x) = e^{- \ep R^{-2} (\log R)^{-1} |x-y|} \psi(x)
\end{equation*}
satisfies
\begin{equation*}
\begin{aligned}
H \rho_y(x) & \geq e^{-\ep R^{-2} (\log R)^{-1} |x-y|} (H \psi (x) - C \ep R^{-2}) \\
& \geq e^{-\ep R^{-2} (\log R)^{-1} |x-y|} (c R^{-2} - C \ep R^{-2}) \\
& \geq c R^{-2} \id_{\{ y \}}.
\end{aligned}
\end{equation*}
Since the potential $V$ is non-negative, the Hamiltonian $H$ has a comparison principle.  Since $\rho_y \geq 0$, we conclude that $C R^2 \rho_y$ is a supersolution of the equation solved by $x \mapsto H_Q^{-1}(x,y)$.  In particular, $0 \leq H_Q^{-1}(x,y) \leq C R^2 \rho_y(x) \leq e^{C L - c L^{-1}|x-y|}$.
\end{proof}

\section{Multiscale analysis}
\label{s.multiscale}

We now assemble our ingredients into a proof of \tref{main} by following the outline of \cite{Bourgain-Kenig}.  In this section, we assume that, unless otherwise specified, all squares have dyadic side length and are half-aligned.  That is, all squares have the form
\begin{equation*}
Q = x + [0,2^n)^2 \cap \Z^2 \quad \mbox{for } x \in 2^{n-1} \Z^2.
\end{equation*}

We need a simple covering lemma.

\begin{lemma}
\label{l.cover}
If $K \geq 1$ an integer, $\alpha \geq C^K$ a dyadic integer, $L_0 \geq \alpha L_1 \geq L_1 \geq \alpha L_2 \geq L_2$ dyadic scales, $Q \subseteq \Z^2$ an $L_0$-square, and $Q''_1, ..., Q''_K \subseteq Q$ are $L_2$-squares, then there is a dyadic scale $L_3 \in [L_1,\alpha L_1]$ and disjoint $L_3$-squares $Q'_1, ..., Q'_K \subseteq Q$ such that,
\begin{equation}
\label{e.coverproperty}
\mbox{for every $Q''_k$, there is $Q'_j$, such that $Q''_k \subseteq Q'_j$ and $\dist(Q''_k,Q \setminus Q'_j) \geq \tfrac18 L_3$}.
\end{equation}
Note that the inequality in \eref{coverproperty} guarantees that $Q''_k$ is far from the part of the boundary of $Q'_j$ that intersects $Q$.
\end{lemma}

\begin{proof}
Start with $L_3 = L_1$ and select any list of $L_3$-squares $Q'_1, ..,. Q'_K \subseteq Q$ so that \eref{coverproperty} holds.  Initially, the $Q'_k$ may not be disjoint.  We modify this family, decreasing the size of the family while increasing the size of the squares.  We iterate the following: If $Q'_j \cap Q'_k \neq \emptyset$ for some $j < k$, then we delete $Q'_k$ from the list and increase the size of all the squares by a constant (universal) factor to maintain \eref{coverproperty}.  This process must stop after at most $K-1$ stages.  Thus, having $\alpha \geq C^K$ is enough room to find a scale $L_3$ that works. Finally, let $Q'_{K'+1}, ..., Q'_K \subseteq Q$ be any additional $L_3$-squares such that $Q'_1, ..., Q'_K$ are disjoint.
\end{proof}

\tref{main} is implied by the following multiscale analysis.

\begin{theorem}
\label{l.main2}
For every $\gamma \in (\tfrac14,\tfrac12)$, there are
\begin{enumerate}
\item small $1 > \ep > \nu > \delta > 0$
\item integer $M \geq 1$
\item dyadic scales $L_k$, for $k \geq 0$, with $\lfloor \log_2 L_k^{1-6\ep} \rfloor = \log_2 L_{k-1}$
\item decay rates $1 \geq m_k \geq L_k^{-\delta}$
\item random sets $F_k \subseteq F_{k+1} \subseteq \Z^2$
\end{enumerate}
such that
\begin{enumerate}
\setcounter{enumi}{5}
\item $F_k$ is $\eta_k$-regular in $Q$ for $\ell(Q) \geq L_k$, where $\eta_k = \ep^2 + L_0^{-\ep} + \cdots + L_k^{-\ep} < \ep$.
\item $F_k \cap Q$ is $V_{F_{k-1} \cap 2Q}$-measurable for $\ell(Q) \geq L_k$
\item if $\ell(Q) = L_k$, $0 \leq \bar \lambda \leq e^{-L_M^\delta}$, and $\mathcal E_g(Q)$ denotes the event that
\begin{equation*}
|(H_Q - \bar \lambda)^{-1}(x,y)| \leq e^{L_k^{1-\ep} - m_k |x-y|} \quad \mbox{for } x, y \in Q
\end{equation*}
holds, then
\begin{equation*}
\P[\P[\mathcal E_g(Q) | V_{F_k \cap Q}] = 1] \geq 1 - L_k^{-\gamma}.
\end{equation*}
\item $m_k \geq m_{k-1} - L_k^{-\nu}$ for $k > M$.
\end{enumerate}
\end{theorem}

\begin{proof}
{\bf Step 1.} We set up the base case.  Assume $\ep, \nu, \delta, M, L_k$ are as in (1-3).  We impose constraints on these objects during the proof.  For the base case, we define $F_k = \lceil \ep^{-2} \rceil \Z^2$ and $m_k = L_k^{-\delta}$ for $k = 0, ..., M$.

\begin{claim}
If $L_0 \geq C_{\ep,\delta}$, then (1-9) hold for $k = 0, ..., M$.
\end{claim}

Fix $k = 0, ..., M$.  Every ball $B_{L_k^{\delta/3}}(x)$ contains $c \ep^4 L_k^{2 \delta / 3}$ elements of $F_k$.  A union bound gives
\begin{equation*}
\P\left[\max_{x \in Q} \min_{y \in Q \cap F_k \cap \{ V = 1 \}} |x-y| \geq C L_k^{\delta/3}\right] \leq L_k^2 e^{-c \ep^4 L_k^{2\delta/3}}.
\end{equation*}
Thus, by \lref{principal} and \lref{continuity}, we see that, for every $0 \leq \bar \lambda \leq e^{-L_M^\delta} \leq e^{-L_k^\delta}$ and $L_k$-square $Q$,
\begin{equation*}
\P[\P[\mathcal E_g(Q)|V_{F_k \cap Q}] = 1] \geq 1 - e^{-L_k^{\delta/3}} \geq 1 - L_k^{-\gamma}.
\end{equation*}
In particular, (8) holds.  Since $F_k$ is $\ep^2$-regular in all large squares and making $L_0$ large gives $\ep^2 < \eta_k < \ep$, we see that (6) holds.  The other properties are immediate.

{\bf Step 2.}  We set up the induction step.  We choose $M \geq 1$ so that
\begin{equation*}
L_k^{\delta} \geq L_{k-M} \geq L_k^{\delta/2} \quad \mbox{for } k > M.
\end{equation*}
We call an $L_j$-square $Q$ ``good'' if
\begin{equation*}
\P[\mathcal E_g(Q) | V_{F_j \cap Q}] = 1.
\end{equation*}
That is, an $L_j$-square is good if, after observing the potential on the frozen sites $F_j \cap Q$, we see that $(H_Q - \bar \lambda)^{-1}$ is well-behaved no matter what happens to the potential on $Q \setminus F_j$.  Note that the event that $Q$ is good is $V_{F_j \cap Q}$-measurable.  An $L_j$-square that is not good is ``bad.''  We must control the bad squares in order to apply \lref{multiscale}.

Suppose that $Q$ is an $L_k$-square and we have chain
\begin{equation*}
Q \supseteq Q_1 \supseteq \cdots \supseteq Q_M
\end{equation*}
with $Q_i$ bad and $\ell(Q_i) = L_{k-i}$ for $i = 1, ..., M$.  We call $Q_M$ a ``hereditary bad subsquare'' of $Q$.  Note that the set of hereditary bad subsquares of $Q$ is a $V_{F_{k-1} \cap Q}$-measurable random variable.  We control the number of hereditary bad subsquares using the following claim.

\begin{claim}
If $\ep < c$ and $N \geq C_{M,\gamma,\delta}$, then, for all $k > M$, 
\begin{equation*}
\P[\mbox{$Q$ has fewer than $N$ hereditary bad subsquares}] \geq 1 - L_k^{-1}.
\end{equation*}
\end{claim}

Writing $N = (N')^M$, we can use the induction hypothesis to estimate
\begin{equation*}
\begin{aligned}
& \P[\mbox{$Q$ has more than $N$ hereditary bad subsquares}] \\
& \leq \sum_{\substack{Q' \subseteq Q \\ \ell(Q') = L_j \\ k-M < j \leq k}} \P[\mbox{$Q'$ has more than $N'$ bad $L_{j-1}$-subsquares}] \\
& \leq \sum_{\substack{Q' \subseteq Q \\ \ell(Q') = L_j \\ k-M < j \leq k}} (L_j/L_{j-1})^{CN'} (L_{j-1}^{-\gamma})^{c N'} \\
& \leq \sum_{k-M < j \leq k} (L_k/L_j)^C (L_j/L_{j-1})^{CN'} (L_{j-1}^{-\gamma})^{c N'} \\
& \leq C M L_k^C \max_{k - M < j \leq k} L_{j-1}^{(C\ep - c \gamma) N'} \\
& \leq C M L_k^C (L_k^{(C \ep - c \gamma) N'} + L_k^{(C \ep - c \gamma) \delta N'}).
\end{aligned}
\end{equation*}
The claim follows making $\ep < c$ and $N' \geq C_{M,\gamma,\delta}$.

We now fix an integer $N \geq 1$ as in the claim.  We call an $L_k$-square $Q$ ``ready'' if $k > M$ and $Q$ has fewer than $N$ hereditary bad $L_{k-M}$-subsquares.  Note the event that $Q$ is ready is $V_{F_{k-1} \cap Q}$-measurable.

Suppose the $L_k$-square $Q$ is ready.  Let $Q'''_1, ..., Q'''_N \subseteq Q$ be a list of $L_{k-M}$-squares that includes every hereditary bad $L_{k-M}$-subsquare of $Q$.  Let $Q''_1, ..., Q''_N \subseteq Q$ be a list of $L_{k-1}$-squares that includes every bad $L_{k-1}$-subsquare of $Q$.  (Since each bad $L_{k-1}$ subsquare contains at least one hereditary bad $L_{k-M}$-subsquare, their number is also bounded by $N.$)  Applying \lref{cover}, we can choose a dyadic scale $L' \in [c_N L_k^{1-2\ep}, L_k^{1-2\ep}]$ and disjoint $L'$-squares $Q'_1, ..., Q'_N \subseteq Q$ such that, for every $Q''_i$, there is a $Q'_j$ such that $Q''_i \subseteq Q'_j$ and $\dist(Q''_i, Q \setminus Q'_j) \geq \tfrac18 L'$.  Note that we can choose $Q'_i, Q''_i, Q'''_i$ in a $V_{F_{k-1} \cap 2Q}$-measurable way.

For $k > M$, we define $F_k$ to be the union of $F_{k-1}$ and the subsquares $Q'_1, ..., Q'_K \subseteq Q$ of each ready $L_k$-square $Q$.  For $k > M$, we define $m_k = m_{k-1} - L_k^{-\nu}$.

{\bf Step 3.}  Having verified properties (1-9) for $k = 0, ..., M$, we now verify properties (1-9) for $k > M$ by induction.  Note that (1-5) and (9) are automatic from the definitions.  We must verify (6-8) for $k > M$, assuming (1-9) holds for all $j < k$.

\begin{claim}
Properties (6) and (7) hold.
\end{claim}

For each $L_k$-square $Q$, the event that $Q$ is ready, the scale $L'$, and the $L'$-squares $Q'_i \subseteq Q$ are all $V_{F_{k-1} \cap Q}$ measurable.  Thus, $F_k \cap Q$ is $V_{F_{k-1} \cap 2Q}$ measurable.  Note that we have $2Q$ in place of $Q$ because each $L_k$-square $Q$ intersects $8$ other half-aligned $L_k$-squares.  In particular, (7) holds.

To see (6), observe that, for each $L_k$-square $Q$, the set $Q \cap F_k \setminus F_{k-1}$ is covered by at most $9N$ squares $Q'_i$ of size less than $L_k^{1-2\ep}$.  In particular, if $Q''$ is tilted square, $Q \cap F_{k-1}$ is $\eta_{k-1}$-sparse in $Q''$, and $Q \cap F_k$ is not $\eta_k$-sparse in $Q''$, then $Q''$ must intersect one of the $Q'_i$ and have size at most $L_k^{1-\ep}$.  This implies that $F_k \cap Q$ is $\eta_k$-regular in $Q$.

\begin{claim}
If the $L_k$-square $Q$ is ready, $|\lambda - \bar \lambda| \leq e^{-L_{k-1}^{1-\ep}}$, and $H_{Q'_i} \psi = \lambda \psi$, $E = Q'_i \setminus \cup_j Q''_j$, and $G = Q'_i \cap \cup_j Q'''_j$, then
\begin{equation*}
e^{cL_{k-1}^{1-\delta}} \| \psi \|_{\ell^\infty(E)} \leq \| \psi \|_{\ell^2(Q'_i)} \leq (1 + e^{-c L_{k-M}^{1-\delta}}) \| \psi \|_{\ell^2(G)}.
\end{equation*}
\end{claim}

If $x \in Q_i' \setminus G$, then there is a $j = 1, ..., M$ and a good $L_{k-j}$-square $Q'' \subseteq Q'_i$ with $x \in Q''$ and $\dist(x, Q'_i \setminus Q'') \geq \tfrac18 L_{k-j}$.  Moreover, if $x \in E$, then $j = 1$.  By the definition of good and \lref{continuity},
\begin{equation*}
|\psi(x)| \leq 2 e^{L_{k-j}^{1-\ep} - \tfrac18 m_{k-j} L_{k-j}} \| \psi \|_{\ell^2(Q'_i)} \leq e^{- c L_{k-j}^{1-\delta}} \| \psi \|_{\ell^2(Q'_i)}.
\end{equation*}
In particular, we see that
$$ \| \psi \|_{\ell^\infty(E)} \leq e^{-c L_{k-1}^{1-\delta}} \| \psi \|_{\ell^2(Q'_i)}$$
and
$$ \| \psi \|_{\ell^\infty(Q_i' \setminus G)} \leq e^{-c L_{k-M}^{1-\delta}} \| \psi \|_{\ell^2(Q'_i)}.$$
Together these imply the claim.

\begin{claim}
If $Q$ is an $L_k$-square and $\mathcal E_i(Q)$ denotes the event that
\begin{equation*}
\mbox{$Q$ is ready and $\P[ \| (H_{Q'_i} - \bar \lambda)^{-1}\| \leq e^{L_k^{1-4 \ep}} | V_{F_k \cap 4Q}] = 1,$}
\end{equation*}
then $\P[\mathcal E_i(Q)] \geq 1 - L_k^{C \ep - 1/2}$.
\end{claim}

Recall the event $Q$ ready and squares $Q'_i \subseteq Q$ are $V_{F_{k-1} \cap 2Q}$-measurable.  We may assume $i = 1$.  We apply \lref{wegner} to the square $Q'_1$ with scales $L' \geq L_k^{1-4\ep} \geq L_k^{1-4\ep} \geq L_{k-1} \geq L_{k-1}^{1-\delta} \geq L_{k-1}^{1-\ep}$, frozen set $F_{k-1}$, defects $\{ Q''_j : Q''_j \subseteq Q'_1 \}$, and $G = \cup \{ Q'''_j : Q'''_j \subseteq Q'_1 \}$.  Assuming $\ep > 2 \delta$, the previous claim provides the localization required to verify the hypotheses of \lref{wegner}.  Since $Q'_1 \subseteq F_k$ when $Q$ is ready, the claim follows.

\begin{claim}
If $Q$ is an $L_k$-square and $\mathcal E_1(Q), ..., \mathcal E_N(Q)$ hold, then $Q$ is good.
\end{claim}

We apply \lref{multiscale} to the square $Q$ with small parameters $\ep > \nu > 0$, scales $L_k \geq L_k^{1-\ep} \geq L_k^{1-2\ep} \geq L_k^{1-3\ep} \geq L_k^{1-4\ep} \geq L_{k-1} \geq L_{k-1}^{1-\ep}$, and defects $Q'_1, ..., Q'_N$.  We conclude that
\begin{equation*}
|(H_Q - \bar \lambda)^{-1}(x,y)| \leq e^{L_k^{1-\ep} - m_k |x-y|}.
\end{equation*}
Since the events $\mathcal E_i(Q)$ are $V_{F_k \cap 2Q}$-measurable, we see that $Q$ is good.

\begin{claim}
Property (8) holds.
\end{claim}

Combining the previous two claims, for any $L_k$-square $Q$, we have $\P[\mathcal E_g(Q)] \geq 1 - N L_k^{C \ep - 1/2} \geq 1 - L_k^{-\gamma}$, provided that $\gamma < 1/2 - C \ep$.
\end{proof}

\end{document}